\documentclass[12pt]{article}
\usepackage{graphicx} % Required for inserting images
\usepackage{hyperref}
\usepackage{amsmath,amssymb,amsfonts,amsthm,tikz,xspace,fullpage}
\usepackage{authblk}
\usetikzlibrary{calc}
\usetikzlibrary{arrows.meta}
\title{The burning game on graphs}
\author[1]{Nina Chiarelli}
\author[2,4]{Vesna Ir\v si\v c Chenoweth}
\author[3,4]{Marko Jakovac}
\author[5]{William B. Kinnersley}
\author[6]{Mirjana Mikala\v{c}ki} 
\affil[1]{FAMNIT and IAM, University of Primorska, Slovenia}
\affil[2]{Faculty of Mathematics and Physics, University Ljubljana, Slovenia}
\affil[3]{Faculty of Natural Sciences and Mathematics, University of Maribor, Slovenia}
\affil[4]{Institute of Mathematics, Physics and Mechanics, Ljubljana, Slovenia}
\affil[5]{Department of Mathematics and Applied Mathematical Sciences, University of Rhode Island, USA}
\affil[6]{Department of Mathematics and Informatics, Faculty of Sciences, University of Novi Sad, Serbia}

\date{\today}

\newtheorem{theorem}{Theorem}

\newtheorem{corollary}[theorem]{Corollary}
\newtheorem{lemma}[theorem]{Lemma}

\newtheorem{question}[theorem]{Question}

\newtheorem{proposition}[theorem]{Proposition}

\theoremstyle{definition}
\newtheorem{example}[theorem]{Example}

\newcommand{\bg}{b_{\rm g}}
\newcommand{\ceil}[1]{\left\lceil #1 \right\rceil}
\newcommand{\floor}[1]{\left\lfloor #1 \right\rfloor}
\newcommand{\size}[1]{\left \vert #1 \right \vert}
\newcommand{\dist}{\textrm{dist}}

\newcommand{\fastplayer}{Burner\xspace}
\newcommand{\slowplayer}{Staller\xspace}
\DeclareMathOperator{\rad}{rad}
\DeclareMathOperator{\diam}{diam}
\DeclareMathOperator{\ecc}{ecc}
\DeclareMathOperator{\CL}{CL}

\begin{document}

 \maketitle

\begin{abstract}
    Motivated by the burning and cooling processes, the burning game is introduced. Two players (Burner and Staller) play the game on a graph $G$ by alternately selecting vertices of $G$ to burn; as in the burning process, burning vertices spread fire to unburned neighbors.  Burner aims to burn all vertices of $G$ as quickly as possible, while Staller wants the process to last as long as possible.  If both players play optimally, then the number of time steps needed to burn the whole graph $G$ is the game burning number $b_{\rm g}(G)$ if Burner makes the first move, and the Staller-start game burning number $b_{\rm g}'(G)$ if Staller starts.

    In this paper, basic bounds on $b_{\rm g}(G)$ are given and several fundamental properties of the burning game established. Graphs with small game burning numbers are characterized and the game is studied on paths and cycles. An analogue of the burning number conjecture for the burning game is also considered. Finally, it is shown that the problem of determining whether or not $b_{\rm g}(G) \le k$ is NP-hard. 
\end{abstract}


\medskip\noindent
{\bf AMS Subj.\ Class: 05C57, 05C78}

\section{Introduction}

The processes of virus propagation in medicine or in computer networks, or the spread of the trends over social networks, are just some examples of various natural and engineered phenomena spread over networks, that are active research topics (see, e.g.~\cite{banerjee2014epidemic, kephart1992directed, kramer2014experimental, rogers2014diffusion}). In all the aforementioned problems, the natural question is how quickly the contagion can spread over all the members in the network. 

The \textit{burning process} on graphs was introduced in~\cite{bonato2016burn} as a simplified deterministic model to analyze this question, inspired by the processes of firefighting~\cite{barghi2015firefighting,finbow2009firefighter}, graph cleaning~\cite{alon2009cleaning} and graph bootstrap percolation~\cite{balogh2012graph}. (Unknown to these authors, a similar process was actually introduced much earlier in the paper of Alon~\cite{alon1992transmitting}.) Later on, the \textit{cooling process} was introduced in~\cite{bonato2024cool} as a dual of the burning process, modeling the mitigation of infection spread and virus propagation. 

The burning and cooling processes are defined as follows (merged together from ~\cite{bonato2021burning, bonato2016burn, bonato2022book, bonato2024cool}). For a finite, simple, undirected graph $G$, the burning process on $G$ is a discrete-time process, in which each vertex is in one of two states: unburned or burned. Initially, at time step $t = 0$, all vertices are unburned. At each subsequent step $t \geq 1$ (which we also call a round), if there are still unburned vertices, one of them is selected to burn -- this vertex is referred to as a \textit{source}. Once a vertex is burned, it remains in this state for the remainder of the process. Moreover, every burned vertex at time $t$ spreads the fire to all of its unburned neighbors at time $t + 1$. The process continues until every vertex of $G$ is burned. The \textit{burning number} $b(G)$ of a graph $G$ is the smallest number of rounds (or time steps) needed for the process to end. Analogously, the \textit{cooling number} $CL(G)$ of $G$ is the largest number of rounds for the cooling process to end. The sequence of sources chosen in an instance of the burning process (respectively, cooling process) is referred to as a \textit{burning sequence} (resp. \textit{cooling sequence}. The length of the shortest burning sequence is $b(G)$, and for every graph $G$ it holds that $b(G)\le CL(G)$. Note that in the burning process, the selection of a new source and the spread of fire to neighboring vertices happen simultaneously. Thus, there is a source selected in every round, even if the selected vertex would have burned in the same round anyway. Note however that in the cooling process it can happen that there is no source selected in the last round if all the remaining vertices of the graph already burned.

Although it was only recently introduced, graph burning has stimulated a great deal of research.  Much of this research has focused on resolving the so-called \textit{burning number conjecture} posed by Bonato et al.\ in \cite{bonato2016burn}, which asserts that every $n$-vertex connected graph $G$ satisfies $b(G) \le \ceil{\sqrt{n}}$.  Upper bounds on $b(G)$ have been gradually improved over time by several authors (see e.g. \cite{BBBCKPR23}, \cite{BBJRR18}, \cite{LL16}, and \cite{NT24}). It is known that if the burning number conjecture holds for trees, then it holds for all connected graphs; thus, several papers have focused on determining or bounding $b(G)$ for various classes of trees, such as spiders (\cite{BL19}, \cite{DDSSS18}) and caterpillars (\cite{HTK20}, \cite{LHH20}).  Several authors have also investigated the computational complexity of determining the burning number of a graph; this problem was shown to be NP-complete by Bessy et al.\ in \cite{bessy2017burning}, although polynomial-time approximation algorithms are known for several classes of graphs (see \cite{bessy2017burning}, \cite{BK19}, \cite{BL19}).  For more details on previous work in the area, see the recent survey \cite{bonato2021burning}.

Motivated by the two aforementioned processes of burning and cooling, we introduce a new graph game - the \emph{burning game}.  In the burning game on a graph $G$, the two players \fastplayer (he/him) and \slowplayer (she/her) take turns selecting vertices of $G$ to burn; as in the burning process, burning vertices spread fire to unburned neighbors.  Burner aims to burn all vertices of $G$ as quickly as possible, while Staller wants the process to last as long as possible.  The burning game is similar in spirit to several other competitive games based on graph parameters, e.g.~domination games \cite{bresar2021book, portier2022proof34conjecturetotal}, the coloring game \cite{bodlaender1991coloring, zuev2015coloring}, the competition-independence game \cite{phillips2001independence, goddard2018independence} and saturation games~\cite{hefetz2016saturation, keusch2016colorability}. 

Formally, the game is defined as follows. Let $G$ be a finite simple graph. Vertices are burned or unburned, but once burned they stay in this state until the end of the game. At time step $t=0$, all vertices are unburned. In each time step $t \geq 1$, first all neighbors of burned vertices become burned (\emph{the spreading phase}), and then one of the players burns one unburned vertex in this time step as well (\emph{the selection phase}). The game ends in the first time step $t$ in which all vertices of $G$ are burned. The aim of \fastplayer is to minimize the total number of time steps and the aim of \slowplayer is to maximize it. If both players play optimally, then the number of time steps needed to burn the whole graph $G$ is called the \emph{game burning number} $\bg(G)$ if \fastplayer makes the first move, and the \emph{\slowplayer-start game burning number} $\bg'(G)$ if \slowplayer starts.

One time step represents \emph{one round} in the burning game, whose first part is the spreading phase, and the second part is the selection phase. The first round in the game consists only of the selection phase. However, it is possible that the last round ends after only the spreading phase, if there are no unburned vertices left in the graph. 

The burning number of a graph, $b(G)$, 
can equivalently be viewed as the length of the burning game in which \fastplayer is the only player, while the cooling number $\CL(G)$ 
can equivalently be seen as the length of the burning game where \slowplayer is the only player. Just note that due to our specification of a round into two phases, the burning game with only \fastplayer playing can slightly differ from the burning processes (in the last round).

\subsection{Organization of the paper} The rest of the paper is organized as follows. In the following subsection, we state some notation that will be used throughout the rest of the paper. In Section~\ref{sec:basic}, we state some basic bounds on $\bg(G)$, establish some properties of the burning game that will be useful later in the paper, and give structural characterizations of graphs with small game burning numbers. In Section~\ref{sec:upper-bound}, we consider an analogue of the burning number conjecture for the burning game.  In Section~\ref{sec:complexity}, we consider the computational complexity of the burning game; in particular, we prove that the problem of determining whether or not $\bg(G) \le k$ is NP-hard. Lastly, Section~\ref{sec:questions} contains some concluding remarks and open problems.
 
\subsection{Notation} We use standard graph theory notation throughout. 
For a given graph $G$, we denote its vertex set and edge set by $V(G)$ and $E(G)$, respectively. 
Given graphs $H, F$, we write $H\subseteq F$ to denote that $H$ is a \emph{subgraph} of $F$, meaning that $V(H)\subseteq V(F)$ and $E(H)\subseteq E(F)$. 
Given any set $S\subseteq V(G)$, we denote the (joint) \textit{external neighborhood} of $S$ by $N_G(S)=\{u \in V(G)\setminus S : ux \in E(G), x\in S\}$; when the graph $G$ is clear from context, we simply write $N(S)$. The \textit{closed neighborhood} of $S$ in $G$ is $N_G(S) \cup S$, and is denoted by $N_G[S]$ (or just $N[S]$).  When $S$ consists of only one vertex, $x$, we write $N_G(x)$ (or $N(x)$) for the external neighborhood and $N_G[x]$ (or $N[x]$) for the closed neighborhood. The distance between vertices $u$ and $v$ is denoted by $d_G(u, v)$ (or $d(u,v)$). If $v \in V(G)$ and $k \geq 0$, then the \textit{$k$th closed neighborhood} $N_k[v]$ of $v$ is defined as the set of all vertices within distance $k$ of $v$ (including $v$).
We let $deg_G(x) = |N_G(x)|$ denote the degree of vertex $x$ in graph $G$; once again, when $G$ is clear from context, we simply write $deg(x)$. The maximum and minimum vertex degree in a given graph $G$ are denoted by $\Delta(G)$ and $\delta(G)$, respectively. 
If $G$ is a graph and $u$ is a vertex of $G$, then the \textit{eccentricity} of $u$ is defined as $\ecc(v)= \max\{d(u,v): v\in V(G)\}$. The \textit{radius} and \textit{diameter} of $G$ are defined as the minimum and maximum eccentricities, respectively, over all vertices in $G$.
We denote with $G^2$ the \textit{square} of a graph $G$, i.e., the vertex set of $G^2$ is $V(G^2)=V(G)$ and two vertices $u,v \in V(G^2)$ are adjacent if and only if $d(u,v) \le 2$.

\section{Basic properties}
\label{sec:basic}

In this section, we establish some basic properties of the burning game, as well as several elementary bounds on $\bg$ that will be useful throughout the remainder of the paper.

As might be expected, the game burning number of a graph is closely connected to its burning number and its cooling number.  Our first result formalizes this connection.

\begin{proposition}
    \label{prop:trivial-bounds}
    If $G$ is a connected graph, then $b(G) \leq \bg(G) \leq \min\{ \CL(G), 2 b(G^2) - 1\}$ and $b(G) \leq \bg'(G) \leq \min\{ \CL(G), 2 b(G^2)\}$.
\end{proposition}

\begin{proof}
    Regardless of who plays first, let $x_1, x_2, \ldots, x_k$ be a sequence of moves in the burning game. Selecting $x_1, x_2, \ldots$ as sources in the graph burns the whole graph in $\bg(G)$ number of rounds, thus $\bg(G) \geq b(G)$; similarly $\bg(G) \leq \CL(G)$. 

    To see that $\bg(G) \leq 2 b(G^2) - 1$, consider the following strategy for Burner. Let $b(G^2) = k$ and let $x_1, \dots, x_k$ be an optimal sequence of sources for the burning process on $G^2$. Burner's strategy in the burning game is to play vertex $x_i$ on his $i$th turn or, if $x_i$ is already burned, to play any unburned vertex. By choice of the $x_i$, every vertex $v$ in $G$ must be within distance $k-i$ of some $x_i$ in $G^2$, thus $v$ must be within distance $2(k-i)$ of some $x_i$ in $G$.  Burner's strategy ensures that vertex $x_i$ is burned no later than round $2i-1$; since the fire will reach $v$ within the next $2(k-i)$ rounds of the game, $v$ will burn no later than round $2 k - 1$.

    A similar argument shows that $\bg'(G) \le 2b(G^2)$; the only change is that Burner plays $x_i$ in round $2i$ (provided that it is not already burned by then).
\end{proof}

Note that since $G$ is a spanning subgraph of $G^2$, any sequence of moves that burns $G$ would also burn $G^2$; hence $b(G^2) \leq b(G)$. Thus, as a consequence of Proposition \ref{prop:trivial-bounds} we have $\bg(G) \le 2b(G)-1$ and $\bg'(G) \le 2b(G)$; these bounds are sometimes more convenient than the bounds involving $b(G^2)$. Note also that the use of minimum in the upper bound is needed, as for example, $CL(K_n) = 2 < 3 = 2 b(K_n^2) - 1$ and $CL(P_n) = \ceil{\frac{n+1}{2}} > 2 \ceil{\sqrt{n}} -1 \geq  2 b(P_n^2) - 1$ (see \cite{bonato2024cool} and \cite{bonato2016burn}).

As with the burning number, the game burning number of a graph $G$ can be bounded above in terms of the radius of $G$, as we next show.

\begin{proposition}
    \label{prop:radius}
    If $G$ is a connected graph, then $\bg(G) \leq \rad(G) + 1$ and $\bg'(G) \leq \min\{\rad(G) + 2, \diam(G)+1\}$.
\end{proposition}

\begin{proof}
    In the Burner-start game, Burner can burn a central vertex $v$ in round 1, then play arbitrarily for the rest of the game.  Since all vertices are within distance $\rad(G)$ of $v$, they all burn by round $\rad(G)+1$; hence $\bg(G) \le \rad(G)+1$. In the Staller-start game, Burner can likewise burn $v$ in round 2 (provided it is not already burned) to ensure that all vertices burn within $\rad(G)+2$ rounds, so $\bg'(G) \le \rad(G)+2$.  Finally, $\bg'(G) \le \diam(G)+1$ because no matter which vertex Staller burns in round 1, all vertices will have burned by round $\diam(G)+1$.
\end{proof}

We conclude this section with another general upper bound for the game burning number. 

\begin{proposition}
    \label{prop:Delta}
    Let $G$ be a graph on $n$ vertices. If $\Delta(G) \leq n-2$, then $\bg(G) \leq n - \Delta(G)$, and if $\Delta(G) \leq n-3$, then $\bg'(G) \leq n - \Delta(G)$.
\end{proposition}
\begin{proof}
    The upper bound for the Burner-start game follows from an analogous result for the burning number given in \cite{bonato2016burn}. Burner's strategy is to start the game by playing a vertex $v$ with maximum degree and play any unburned vertex available in the next turns. Since during the spreading phase of round 2 all neighbors of $v$ burn, these vertices can never be played during the game. Thus at most $n - \Delta(G)$ vertices can be played. If the game ends with one of the players playing the last remaining vertex, then clearly $\bg(G) \leq n - \Delta(G)$. Now suppose that the game ends during the spreading phase of round $k$ (so $\bg(G) \leq k$). Since $v$ is not a universal vertex (as $\Delta(G) \neq n-1$), we must have $k \ge 3$.  Since all neighbors of $v$ burn during the spreading phase of round 2, it must be that some vertex in $V(G) - N[v]$ burns during the spreading phase of round $k$.  Hence at least $\Delta(G)+1$ vertices burn during spreading phases, so at most $n - \Delta(G) - 1$ vertices are played during the game. Thus $k \leq (n - \Delta(G) - 1) + 1 = n - \Delta(G)$.

    Let $v \in V(G)$ be a vertex with maximum degree. In the Staller-start game, Burner's strategy is the following. If Staller starts the game on $N[v]$, then Burner's first move is a vertex from $V(G) - N[v]$ (note that this set is not empty since $v$ is not an universal vertex). This ensures that only one vertex is played on $N[v]$ during the game as they all burn during or before the spreading phase of round 3. However, if Staller's first move is not on $N[v]$, then Burner plays $v$, and again only one move is played on $N[v]$. Thus $\bg'(G) \leq \max\{3, n - \Delta(G)\} = n - \Delta(G)$.
\end{proof}

\subsection{Continuation Principle}
\label{sec:continuation-principle}

In this subsection, we present a fundamental result regarding the behavior of the burning game that greatly simplifies many arguments.  

When analyzing the burning game, we would often like to consider a game that is ``already in progress'' -- that is, with some vertices already burning.  Given a graph $G$ and $B \subseteq V(G)$, we let $G \vert B$ denote the graph $G$, with the understanding that the vertices in $B$ are already burning prior to the start of the game.  We refer to the burning game on $G \vert B$ as the burning game \textit{relative to $B$}, and we denote the number of rounds needed to burn all of $G \vert B$ -- assuming that both players play optimally -- by $\bg(G \vert B)$ if Burner makes the first move, and by $\bg'(G \vert B)$ if Staller makes the first move.

The following result, known as the \textit{Continuation Principle}, formalizes the intuition that starting the game with additional vertices burned can never increase the length of the game. (The name Continuation Principle for this result is taken from that of an analogous result for the domination game; see \cite{KWZ13}.)

\begin{theorem}[Continuation Principle]
    \label{thm:continuation-principle}
    If $A \subseteq B \subseteq V(G)$, then $\bg(G|B) \leq \bg(G|A)$ and $\bg'(G|B) \leq \bg'(G|A)$.
\end{theorem}

\begin{proof}
    We provide a strategy of Burner on $G|B$ (under the assumption that Staller is playing optimally); we call this the \textit{real} game. Burner imagines that a game is being played on $G|A$ simultaneously, and we call this the \textit{imagined} game. Burner will play optimally in the imagined game, and will use his strategy for that game to guide his play in the real game.  The set of burned vertices after time step $t$ is denoted by $R(t)$ and $I(t)$ in the real and the imagined game, respectively. Burner will ensure that the following invariant is preserved after each time step: $I(t) \subseteq R(t)$. Since $A \subseteq B$, this holds for $t=0$. 

    Fix some $t \ge 0$ and consider the state of the game after time step $t$.  Since $I(t) \subseteq R(t)$, we have $N[I(t)] \subseteq N[R(t)]$, so after the spreading phase of round $t+1$, every vertex that is burned in the imagined game is also burned in the real game.  If it is Burner's turn to select a vertex in the selection phase, then he considers the imagined game and finds his optimal move $x$ there. If $x$ is unburned in the real game as well, he copies the move there; otherwise, he plays $x$ in the imagined game but an arbitrary unburned vertex $y$ in the real game.  In either case we have 
    \[R(t+1) \supseteq N[R(t)] \cup \{x\} \supseteq N[I(t)] \cup \{x\} = I(t+1),\]
    so the invariant is preserved after time step $t+1$.

    Suppose instead that it is Staller's turn to select a vertex in the selection phase, and Staller makes a move $y$ in the real game.  Since the invariant holds before this move, $y$ is also a legal move in the imagined game. Thus Burner copies the move $y$ to the imagined game as the move of Staller, preserving the invariant.

    The invariant ensures that the real game finishes at the same time step or before the imagined game. Let $m$ be the number of moves in the game on $G|B$ if Burner is using the above strategy. Then since Staller is playing optimally in the real game and Burner is playing optimally in the imagined game, we have $\bg(G|B) \leq m \leq \bg(G|A)$.

    The proof of the continuation principle for the Staller-start game is analogous.
\end{proof}

One useful consequence of the Continuation Principle is that we do not need to worry about whether or not Burner's moves are legal (i.e. whether or not the vertices Burner selects are still unburned).  Selecting a vertex that is already burned would not enlarge the set of burned vertices, and thus would never be better for Burner than playing an arbitrary unburned vertex.  Thus, when presenting a strategy for Burner, we may allow him to select burned vertices, since doing so would not afford him any additional advantage.

\begin{proposition}
    \label{prop:F-S-start}
    If $G$ is a connected graph, then $|\bg(G) - \bg'(G)| \leq 1$.
\end{proposition}

\begin{proof}
    Let $v \in V(G)$ be an optimal first move for Burner. Then by the Continuation Principle, $\bg(G) = 1 + \bg'(G|N[v]) \leq 1 + \bg'(G)$, thus $\bg(G) - \bg'(G) \leq 1$. Analogously we obtain that $\bg'(G) - \bg(G) \leq 1$.
\end{proof}

All possibilities from Proposition \ref{prop:F-S-start} can be achieved. For $n \geq 2$, $\bg(K_n) = \bg'(K_n) = 2$. For $n \geq 3$, $\bg(K_{1,n}) = 2$ and $\bg'(K_{1,n}) = 3$. For even $n \geq 4$, $\bg(Q_n) = \frac{n}{2} + 2$ and $\bg'(Q_n) = \frac{n}{2} + 1$ (see~\cite[Theorem 29]{arxiv-burning}).

\begin{theorem}
    \label{thm:edge-removal}
    If $G$ is a connected graph and $e$ is any edge in $G$, then $\bg(G) \le \bg(G - e) \leq \bg(G) + 2$ and $\bg'(G) \le \bg'(G - e) \leq \bg'(G) + 2$.
\end{theorem}

\begin{proof}
    We argue that $\bg(G) \le \bg(G - e) \leq \bg(G) + 2$; the proof that $\bg'(G) \le \bg'(G - e) \leq \bg'(G) + 2$ is similar.
    
    Let $e=xy$.
    The lower bound for $\bg(G-e)$ is easy to see, as removing one edge from $G$ cannot decrease the game burning number $\bg(G)$, but can only slow down or stop the spreading process. 
    
    To see the upper bound, suppose that \fastplayer plays by strategy $\mathcal{B}$ on $G-e$ that uses his optimal strategy in $G$, until one of the endpoints of $e$, say $x$, burns. Suppose that this occurs in round $k$. Note that $k \le \bg(G)$ since Burner follows an optimal strategy for the game on $G$. Moreover, suppose that \slowplayer played a vertex in this round; the other case is similar. Let $S$ denote the set of burned vertices at the end of round $k$. Note that since fire has not yet had the opportunity to spread along edge $e$, and since \fastplayer has followed an optimal strategy for the game on $G$ (although \slowplayer might not have), we have $\bg(G) \ge k+ \bg(G \vert S)$, hence $\bg(G \vert S) \le \bg(G)-k$.  
    
    If vertex $y$ burns before Burner's next turn, then the absence of edge $e$ has not impacted the game, since there would have been no opportunity for fire to spread along $e$. If the game ends in the spreading phase of round $k+1$, then $\bg(G-e) = k+1 \le \bg(G)+1$.  Otherwise, in round $k+1$ following $\mathcal{B}$, Burner plays $y$ and then continues using the optimal strategy for $G$. \fastplayer's optimal strategy for $G-e$ can last as long as by playing using $\mathcal{B}$, i.e. $\bg(G-e)$ is at most the number of rounds played by $\mathcal{B}$, denote it by $t_{\mathcal{B}}$. Once again, if the game ends with Burner's move or anytime during round $k+2$, then $\bg(G-e) \le t_{\mathcal{B}} \le k+2 \le \bg(G)+2$. Suppose otherwise, and let $z$ denote \slowplayer's move in round $k+2$. Let $S'$ denote the set of burned vertices at the end of round $k+2$, and note that $S' = N_2[S] \cup N[y] \cup \{z\}$ (with neighborhoods taken with respect to $G-e$).  Now observe that $\bg(G-e \vert S') = \bg(G \vert S')$ because both endpoints of $e$ belong to $S'$, so the presence or absence of $e$ does not impact the game relative to $S'$; additionally, $\bg(G \vert S') \le \bg(G \vert S)$ by the Continuation Principle. 
    Thus 
    \[\bg(G-e) \le k+2+ \bg(G-e \vert S') = k+2+\bg(G \vert S') \le k+2+\bg(G \vert S) \le  k+2+\bg(G)-k = \bg(G)+2,\] as claimed.
\end{proof}

Theorem \ref{thm:edge-removal} states that for any graph $G$ and any $e \in E(G)$, we have $\bg(G-e) \in \{\bg(G), \bg(G)+1, \bg(G)+2\}$.  In fact, for appropriate choices of $G$ and $e$, any of these three values for $\bg(G-e)$ can be achieved.  
 
To see that we can have $\bg(G) = \bg(G-e) = k$ for any $k \ge 2$, choose $n$ such that $\bg(P_n) = \bg(P_{n+1}) = k$, which is possible by Theorem \ref{thm:paths}.  Form a graph $G$ as follows.  Start with a path on vertices $v_1, \dots, v_n$ in order.  Suppose that, in the Burner-start game on this path, an optimal first move for Burner would be vertex $v_i$.  Let $y$ be a neighbor of $v_i$; add an additional vertex $x$ to $G$, along with edges $xv_i$ and $xy$.  Now $\bg(G) = \bg(P_n) = k$, since Burner can play $v_i$ on his first move, thereby burning $x$ before Staller's first turn.  (Note that the presence of edge $xy$ is irrelevant, since both $x$ and $y$ burn in the spreading phase of round 2.)  However, since $G-v_i y = P_{n+1}$, we have $\bg(G-v_iy) = \bg(P_{n+1}) = k = \bg(G)$, as claimed.

We can use a similar argument to construct a graph $G$ with edge $e$ such that $\bg(G) = k$ and $\bg(G-e) = k+1$; the only difference is that we need to choose $n$ such that $\bg(P_n) = k$ and $\bg(P_{n+1}) = k+1$.  As before, $\bg(G) = \bg(P_n) = k$, but now $\bg(G-v_iy) = \bg(P_{n+1}) = k+1$.  Likewise, similar arguments can be used to show that we can have $(\bg'(G),\bg'(G-e)) = (k,k)$ and $(\bg'(G),\bg'(G-e)) = (k,k+1)$.

Showing that we can have $\bg(G-e) = \bg(G)+2$ takes a bit more work.

\begin{example}\label{ex:difference_two}
For the graph $G$ in Figure \ref{fig:difference_two}, we have $\bg(G) = 5$ and $\bg(G-vw) = 7$.  

To ensure that $G$ burns within five rounds, \fastplayer can burn vertex $u$ in round 1, vertex $x$ in round 3, and vertex $y$ in round 5. 
Since every vertex $G$ is within distance 4 of $u$, distance 2 of $x$, or distance 0 of $y$, \fastplayer's moves ensure that the entire graph will be burned by the end of round 5, regardless of how \slowplayer plays.

Now consider the burning game on $G-vw$.  Suppose \fastplayer aims to burn the graph within six rounds; we explain how \slowplayer can thwart him.  First, note that \fastplayer's move in round 1 must belong to $N[u]$.  If \fastplayer plays outside of $N[u]$ in round 1, then \slowplayer can play either $w$ or $y$ in round 2 -- at least one of which is guaranteed to be unburned -- which ensures that $u$ is unburned at the end of round 2, along with all vertices between $u$ and at least five of $u_1, u_2, u_3$, $u_4, u_5,$ and $u_6$ (without loss of generality $u_1, u_2, u_3, u_4, u_5$).  At the end of round 2, $u_1$, $u_2$, $u_3$, $u_4$, and $u_5$ are all distance 5 or greater from all burned vertices, so fire cannot spread from any currently burning vertex to any of these five vertices within rounds 3-6.  Additionally, these five vertices are pairwise distance 8 apart, so no single move played in rounds 3-6 could cause more than one of these vertices to be burned by the end of round 6.  Consequently, if Burner does not play his first move in $N[u]$, then the graph cannot be fully burned by the end of round 6.

Suppose then that Burner plays his first move in $N[u]$.  In round 2, Staller will burn $u_1$.  At the end of round 2, no vertices have been burned in the component of $G-vw-\{v_1,v_2,v_3,v_4\}$ containing $w$.  Arguments similar to those in the preceding paragraph show that if Burner does not burn some vertex in $N[w]$ in round 3, then some $w_i$ will remain unburned by the end of round 6; likewise, if he does not burn some vertex in $N[x]$ in round 3, then some $x_i$ will remain unburned by the end of round 6.  Since no vertex belongs to both $N[w]$ and $N[x]$, at least one vertex will remain unburned by the end of round 6, hence $\bg(G-vw) \ge 7$.

We have argued by the above paragraphs that $\bg(G) \le 5$ and $\bg(G-vw) \ge 7$. By Theorem \ref{thm:edge-removal}, we have $\bg(G - vw) \leq \bg(G) + 2 \le 7$, which implies $\bg(G-vw) = 7$. Moreover, by the same theorem, we have $7 = \bg(G - vw) \leq \bg(G) + 2$, and hence $\bg(G) \ge 5$, which implies $\bg(G) = 5$.
\end{example}

\begin{figure}
\centering
\begin{tikzpicture} [inner sep=0mm, thick,
 smallvertex/.style={draw=black, circle, minimum size=0.011cm},
 vertex/.style={draw=black, fill=black, circle, minimum size=0.2cm},
 xscale=1,yscale=1]

\node (a) at (0, 0) [vertex] {};
\node at (0,0) [above left = 0.2cm] {$u$};
 
\node (a11) at (-0.5, 1) [vertex] {};
\node (a12) at (-0.5, 1.5) [vertex] {};
\node (a13) at (-0.5, 2) [vertex] {};
\node (a14) at (-0.5, 2.5) [vertex] {};
\node at (a14) [above = 0.3cm] {$u_1$};
\node (a21) at (0., 1) [vertex] {};
\node (a22) at (0., 1.5) [vertex] {};
\node (a23) at (0., 2) [vertex] {};
\node (a24) at (0., 2.5) [vertex] {};
\node at (a24) [above = 0.3cm] {$u_2$};
\node (a51) at (0.5, 1) [vertex] {};
\node (a52) at (0.5, 1.5) [vertex] {};
\node (a53) at (0.5, 2) [vertex] {};
\node (a54) at (0.5, 2.5) [vertex] {};
\node at (a54) [above = 0.3cm] {$u_3$};
\node (a31) at (-0.5, -1) [vertex] {};
\node (a32) at (-0.5, -1.5) [vertex] {};
\node (a33) at (-0.5, -2) [vertex] {};
\node (a34) at (-0.5, -2.5) [vertex] {};
\node at (a34) [below = 0.3cm] {$u_4$};
\node (a41) at (0., -1) [vertex] {};
\node (a42) at (0., -1.5) [vertex] {};
\node (a43) at (0., -2) [vertex] {};
\node (a44) at (0., -2.5) [vertex] {};
\node at (a44) [below = 0.3cm] {$u_5$};
\node (a61) at (0.5, -1) [vertex] {};
\node (a62) at (0.5, -1.5) [vertex] {};
\node (a63) at (0.5, -2) [vertex] {};
\node (a64) at (0.5, -2.5) [vertex] {};
\node at (a64) [below = 0.3cm] {$u_6$};

\node (u) at (1.5, 0) [vertex] {};
\node at (u) [above = 0.2cm] {$v$};
\node (u1) at (2.125,-0.25) [vertex] {};
\node at (u1) [below = 0.2cm] {$v_1$};
\node (u2) at (2.75,-0.4) [vertex] {};
\node at (u2) [below = 0.2cm] {$v_2$};
\node (u3) at (3.375,-0.25) [vertex] {};
\node at (u3) [below = 0.2cm] {$v_3$};
\node (v) at (4, 0) [vertex] {};
\node at (v) [above left = 0.2cm] {$w$};
\node (v11) at (3.75, 1) [vertex] {};
\node (v12) at (3.75, 1.5) [vertex] {};
\node at (v12) [above = 0.3cm] {$w_1$};
\node (v21) at (4.25, 1) [vertex] {};
\node (v22) at (4.25, 1.5) [vertex] {};
\node at (v22) [above = 0.3cm] {$w_2$};
\node (v31) at (3.75, -1) [vertex] {};
\node (v32) at (3.75, -1.5) [vertex] {};
\node at (v32) [below = 0.3cm] {$w_3$};
\node (v41) at (4.25, -1) [vertex] {};
\node (v42) at (4.25, -1.5) [vertex] {};
\node at (v42) [below = 0.3cm] {$w_4$};

\node (w1) at (5.5, 0) [vertex] {};
\node (w2) at (7, 0) [vertex] {};

\node (x) at (8.5, 0) [vertex] {};
\node at (x) [above left = 0.2cm] {$x$};
\node (x11) at (8.25, 1) [vertex] {};
\node (x12) at (8.25, 1.5) [vertex] {};
\node at (x12) [above = 0.3cm] {$x_1$};
\node (x21) at (8.75, 1) [vertex] {};
\node (x22) at (8.75, 1.5) [vertex] {};
\node at (x22) [above = 0.3cm] {$x_2$};
\node (x31) at (8.25, -1) [vertex] {};
\node (x32) at (8.25, -1.5) [vertex] {};
\node at (x32) [below = 0.3cm] {$x_3$};
\node (x41) at (8.75, -1) [vertex] {};
\node (x42) at (8.75, -1.5) [vertex] {};
\node at (x42) [below = 0.3cm] {$x_4$};

\node (y1) at (10, 0) [vertex] {};
\node (y2) at (11.5, 0) [vertex] {};
\node (y3) at (13, 0) [vertex] {};
\node at (13,0) [above = 0.2cm] {$y$};

\draw (a14) -- (a13) -- (a12) -- (a11) -- (a) -- (a21) -- (a22) -- (a23) -- (a24);
\draw (a34) -- (a33) -- (a32) -- (a31) -- (a) -- (a41) -- (a42) -- (a43) -- (a44);
\draw (a54) -- (a53) -- (a52) -- (a51) -- (a) -- (a61) -- (a62) -- (a63) -- (a64);
\draw (a) -- (u) -- (v); 
\draw (v12) -- (v11) -- (v) -- (v21) -- (v22);
\draw (v32) -- (v31) -- (v) -- (v41) -- (v42);
\draw (v) -- (w1) -- (w2) -- (x);
\draw (a) -- (u) -- (v); 
\draw (x12) -- (x11) -- (x) -- (x21) -- (x22);
\draw (x32) -- (x31) -- (x) -- (x41) -- (x42);
\draw (x) -- (y1) -- (y2) -- (y3);
\draw (u) -- (u1) -- (u2) -- (u3) -- (v);

\end{tikzpicture}
\caption{The graph $G$ in Example \ref{ex:difference_two} with $\bg(G) = 5$ and $\bg(G - vw) = 7$.}
\label{fig:difference_two}
\end{figure}

The following lemma is a useful consequence of Theorem \ref{thm:edge-removal}.

\begin{lemma}
    \label{lem:spanning-subgraph}
    If $H$ is a spanning subgraph of $G$, then $\bg(G) \leq \bg(H)$ and $\bg'(G) \le \bg'(H)$.  
\end{lemma}

\subsection{Characterizations}
\label{sec:characterizations}

In this subsection, we give structural characterizations of graphs having small game burning numbers.

\begin{proposition}
    \label{prop:characterize-1-2}
    Let $G$ be a connected graph.
    \begin{enumerate}
        \item $\bg(G) = 1$ if and only if $G = K_1$;
        \item $\bg'(G) = 1$ if and only if $G = K_1$;
        \item $\bg(G) = 2$ if and only if $G \neq K_1$ and $\Delta(G) \geq |V(G)| - 2$; and
        \item $\bg'(G) = 2$ if and only if $G \neq K_1$ and $\delta(G) \geq |V(G)| - 2$ (if and only if $G$ is isomorphic to a complete graph on at least two vertices without a (possibly empty) matching).
    \end{enumerate}
\end{proposition}

\begin{proof}
    Let $|V(G)| = n$.
    \begin{enumerate}
        \item If $G = K_1$, then trivially $\bg(G) = 1$; since only one vertex is burned in the first round, clearly $\bg(G) > 1$ if $n \geq 2$.
        \item Same as above.
        \item If $\Delta(G) \geq n-2$, then Burner's strategy is to first play a vertex of degree at least $n-2$. In the second round, all vertices but at most one are burned, thus Staller has no other option but to play the remaining unburned vertex and so $\bg(G) \leq 2$. Since $G \neq K_1$, $\bg(G) \geq 2$.

        Now suppose that $\bg(G) = 2$. Suppose that, in an optimal strategy, Burner plays vertex $v$ in round 1.  The remaining vertices in $N[v]$ burn in the spreading phase of round 2.  It must be that $G - N[v]$ contains at most one vertex, since otherwise the game would not end after Staller's turn in round 2.  It follows that $\Delta(G) \geq \deg(v) \ge n-2$. Since $\bg(G) \neq 1$, $G \neq K_1$.

        \item If $G$ is a graph on at least two vertices with $\delta(G) \geq n-2$, then for every vertex $v \in V(G)$ it holds that $G-N[v]$ contains at most one vertex. Thus no matter what vertex Staller burns in round 1, at most one vertex remains unburned after the spreading phase of round 2, so Burner can finish the game with his ensuing move. Since $G \neq K_1$, $\bg'(G) = 2$.

        If $\bg'(G) = 2$, then $G \neq K_1$ and no matter which vertex Staller plays in round 1 starts, Burner can ensure that the game ends in the second time step. This means that every vertex of $G$ has at most one non-neighbor, thus $\delta(G) \geq n-2$. \qedhere
    \end{enumerate}
\end{proof}

\begin{proposition}
    \label{prop:characterize-3}
    Let $G$ be a connected graph. Then $\bg(G) = 3$ if and only if $\Delta(G) \leq |V(G)| - 3$ and there exists $v \in V(G)$ such that every vertex in $V(G-N[v])$ is adjacent to all but at most one vertex in $V(G-N_2[v])$.
\end{proposition}

\begin{proof}
    Suppose first that $\bg(G) = 3$. Then there exists an optimal first move of Burner, $v \in V(G)$, such that no matter what Staller plays during her subsequent turn, at most one unburned vertex remains after the spreading phase of round $3$. Thus every vertex in $V(G) - N[v]$ -- that is, every possible Staller move -- is adjacent to all but at most one vertex not in $N_2[v]$ (and Burner plays this vertex in round 3). Additionally, since $\bg(G) = 3$, by Proposition \ref{prop:characterize-1-2} we know that $\Delta(G) \leq |V(G)| - 3$. 

    Now suppose that both conditions from the statement are satisfied. Burner's strategy is to first play $v$ in round 1.  No matter where Staller plays in round 2, at most one vertex remains unburned after the spreading phase of round 3; Burner plays this vertex (if it exists) with his ensuing turn.
\end{proof}

\begin{proposition}
    \label{prop:characterize-3'}
    Let $G$ be a connected graph. Then $\bg'(G) \leq 3$ if and only if for every $v \in V(G)$, we have $\Delta(G-N_2[v]) \ge \size{V(G-N_2[v])}-2$.
\end{proposition}

\begin{proof}
First suppose that $\Delta(G-N_2[v]) \ge \size{V(G-N_2[v])}-2$ for every vertex $v \in V(G)$. Staller starts the game on some vertex $v$. Then Burner can finish it in at most three rounds if he chooses one of the vertices of maximum degree in $V(G-N_2[v])$ in his subsequent turn. Because of the condition $\Delta(G-N_2[v]) \ge \size{V(G-N_2[v])}-2$ this leaves at most one unburned vertex after the spreading phase of round $3$. Hence, Staller can only play on this vertex (if at all). Thus $\bg'(G) \leq 3$.

Suppose to the contrary that there exists $v \in V(G)$ such that $\Delta(G-N_2[v]) < \size{V(G-N_2[v])}-2$.  On Staller's first turn, she burns $v$.  Let $w$ be Burner's move in the next round.  After the selection phase of round 3, all vertices in $N_2[v]$ have burned, as have all vertices in $N[w]$.  Since $\Delta(G-N_2[v]) < \size{V(G-N_2[v])}-2$, we have 
\[\size{N_2[v] \cup N[w]} < \size{N_2[v]} + \size{V(G-N_2[v])}-1 = \size{V(G)}-1,\]
so at least two unburned vertices remain.  Hence, Staller cannot be forced to burn the rest of the graph in the selection phase of round 3.
\end{proof}

Note that one can obtain a characterization of graphs with the Staller-start game burning number equal to 3 by combining Propositions \ref{prop:characterize-1-2} and \ref{prop:characterize-3'}.

\begin{proposition}
\label{prop:b=bg_diam2}
If $G$ has diameter at most $2$, then $b(G)=\bg(G)$.
\end{proposition}
\begin{proof}
If the diameter of $G$ is $1$, then $G$ is a complete graph on at least two vertices for which $b(G)=\bg(G)=2$.

Now suppose that the diameter is $2$, hence every burning process, whether it is the usual burning or the burning game, will finish in at least $2$ and in at most $3$ rounds. By Proposition~\ref{prop:characterize-1-2}, $\bg(G)=2$ if and only if $G \neq K_1$ and $\Delta(G) \geq |V(G)| - 2$. We already know that $G \neq K_1$ because of the diameter being $2$, hence $G$ must be a complete graph on at least two vertices without a matching. In such a graph we can burn a vertex that spreads fire to all but at most one vertex in the second round. Hence in the second round we only have at most one vertex to choose for the second source, i.e.\ $b(G)=2$.

In the case of $\bg(G)=3$ the graph $G$ must have $\Delta(G) \leq |V(G)| - 3$. No matter how we select the source in the first round, in the second round, after the spread, at least two vertices will remain unburned. We have to chose one of them for the second source, and the third round has to start in order to burn the remaining vertex (or vertices). Hence, $b(G)=3$.
\end{proof}

\section{The Burning Number Conjecture}
\label{sec:upper-bound}

As mentioned in the Introduction, the \textit{burning number conjecture}, posed by Bonato et al.~\cite{bonato2016burn}, has attracted a great deal of attention in recent years.  Recall that the burning number conjecture states that for every $n$-vertex connected graph $G$, we have $b(G) \le \ceil{\sqrt{n}}$ or, equivalently, $b(G) \le b(P_n)$.  In this section, we consider the analogous question for the burning game: what is the maximum value of $\bg(G)$ among $n$-vertex connected graphs $G$?

It was shown in \cite{bonato2016burn} that whenever $G'$ is a spanning subgraph of $G$, we have $b(G) \le b(G')$.  Consequently, 
\[b(G) \le \min\{b(T) \colon T \text{ is a spanning tree of } G\}.\]
Bonato et al.~showed that, in fact, the inequality above always holds with equality; this result, known as the \textit{Tree Reduction Theorem}, is a fundamental tool in attacking the burning number conjecture.

\begin{theorem}[Tree Reduction Theorem (\cite{bonato2016burn}, Corollary 2.5)]\label{thm:tree_reduction}
For every connected graph $G$, we have 
\[b(G) = \min\{b(T) \colon T \text{ is a spanning tree of } G\}.\]
\end{theorem}

As a consequence of the Tree Reduction Theorem, to determine the maximum burning number of an $n$-vertex connected graph, one needs only determine the maximum burning number of an $n$-vertex tree.  Unfortunately, this result does not extend to the burning game.  Lemma \ref{lem:spanning-subgraph} implies that 
\[\bg(G) \le \min\{\bg(T) \colon T \text{ is a spanning tree of } G\};\]
however, unlike in the burning process, equality need not hold.

\begin{example}\label{ex:tree_reduction_not_true}
For the graph $G$ shown in Figure \ref{fig:tree_reduction_not_true}, we have $\bg(G) = 3$, but for every spanning tree $T$ of $G$, we have $\bg(T) \ge 4$.

\begin{figure}[h!]
\centering
\begin{tikzpicture} [inner sep=0mm, thick,
 smallvertex/.style={draw=black, circle, minimum size=0.011cm},
 vertex/.style={draw=black, fill=black, circle, minimum size=0.2cm},
 xscale=1,yscale=1]

\node (a1) at (-0.75, 1) [vertex] {};
\node at (a1) [left = 0.2cm] {$u_1$};
\node (a2) at (0., 1) [vertex] {};
\node at (a2) [above = 0.2cm] {$u_2$};
\node (a3) at (0.75, 1) [vertex] {};
\node at (a3) [right = 0.2cm] {$u_3$};
\node (b) at (0, 0) [vertex] {};
\node at (b) [right = 0.2cm] {$v$};
\node (c1) at (-0.5, -1) [vertex] {};
\node at (c1) [left = 0.2cm] {$w_1$};
\node (c2) at (0.5, -1) [vertex] {};
\node at (c2) [right = 0.2cm] {$w_2$};
\node (d1) at (-0.5, -2) [vertex] {};
\node at (d1) [left = 0.2cm] {$x_1$};
\node (d2) at (0.5, -2) [vertex] {};
\node at (d2) [right = 0.2cm] {$x_2$};
\node (e1) at (-0.5, -3) [vertex] {};
\node at (e1) [left = 0.2cm] {$y_1$};
\node (e2) at (0.5, -3) [vertex] {};
\node at (e2) [right = 0.2cm] {$y_2$};
\node (f) at (0, -4) [vertex] {};
\node at (f) [below = 0.25cm] {$z$};

\draw (a3) -- (b);
\draw (a1) -- (b) -- (a2);
\draw (b) -- (c1) -- (d1) -- (e1) -- (f);
\draw (b) -- (c2) -- (d2) -- (e2) -- (f);
\draw (d1) -- (e2) -- (e1) -- (d2);
\end{tikzpicture}
    \caption{The graph $G$ in Example \ref{ex:tree_reduction_not_true}}
    \label{fig:tree_reduction_not_true}
\end{figure}

Since $\Delta(G) = 5 \leq |V(G)| - 3$ and since every vertex in $G - N[v] = \{x_1, x_2, y_1, y_2, z\}$ is adjacent to all but at most one vertex in $G-N_2[v] = \{y_1, y_2, z\}$, Proposition \ref{prop:characterize-3} yields $\bg(G) = 3$. 

Now consider the game on $T$, for an arbitrary spanning tree $T$ of $G$.  We first claim that to have any chance of completing the game within 3 rounds, Burner must burn vertex $v$ in round 1.  If Burner burns $x_1$, $x_2$, $y_1, y_2$, or $z$ in round 1, then Staller can burn $u_1$ in round 2; both $u_2$ and $u_3$ will remain unburned after the spreading phase of round 3, and Burner cannot burn both of them.  If instead Burner burns $w_1$ (respectively, $w_2$), then again Staller burns $u_1$ in round 2; this time, after the spreading phase of round 3, vertices $z$ and $x_2$ (resp. $x_1$) remain unburned, and Burner cannot burn both.  Finally, if Burner burns some $u_i$ in round 1, then Staller burns $z$ in round 2; after the spreading phase of round 3 both $x_1$ and $x_2$ remain unburned, and once again Burner cannot burn both.

Suppose, then, that Burner burns $v$ in round 1.  Because $T$ is a spanning tree of $G$, some edge $x_iy_j$, $i,j \in \{1,2\}$, must not be present in $T$; by symmetry, assume $x_1y_2$ is not present.  In round 2, Staller burns $x_1$.  Now after the spreading phase of round 3, both $y_2$ and $z$ remain unburned, and Burner cannot burn both. Thus $\bg(T) \geq 4$.
\end{example}

We do not know of any graphs $G$ for which $\min_T \{\bg(T) - \bg(G)\} \ge 2$ (where the minimization is over all spanning trees $T$ of $G$); this would be an interesting topic for future research.\\

For the burning process, Bonato et al.~\cite{bonato2016burn} showed that $b(P_n) = \ceil{\sqrt{n}}$; hence, the burning number conjecture postulates that $b(G) \le b(P_n)$ for all $n$-vertex connected graphs $G$.  Thus, when seeking the maximum value of $\bg(G)$ over $n$-vertex connected graphs, it is natural to look first at the burning game on paths. Theorem \ref{thm:paths} gives lower and upper bounds on $\bg(P_n)$ and $\bg'(P_n)$ that differ by at most 1. 

\begin{theorem}\label{thm:paths}
For all $n \ge 1$, we have 
\[\ceil{\sqrt{2n+1}-1} \le \bg(P_n) \le \ceil{\sqrt{2n+\frac{1}{4}}-\frac{1}{2}}\]
and
\[\ceil{\sqrt{2n+2}-1} \le \bg'(P_n) \le \ceil{\sqrt{2n+\frac{1}{4}}-\frac{1}{2}}.\]
\end{theorem}
\begin{proof}
Before giving the proof, we briefly describe the idea. For the upper bounds, we show how Burner can place sources so that the fire quickly covers large portions of the path and the remaining unburned parts become shorter and easier to handle in the following rounds. For the lower bounds, we describe a strategy for Staller that maintains several separated components of burned vertices, which limits how much the fire can grow from round to round and therefore forces the game to last longer. Both parts of the argument are then formalized using induction.

For the upper bound on $\bg(P_n)$, we first claim that in the Burner-start game, Burner can force all vertices of $P_n$ to burn within $r$ rounds whenever $n \le r(r+1)/2$, even if the impact of Staller's moves is completely ignored.  (Note that by the Continuation Principle, ignoring Staller's moves would be suboptimal for Burner and hence cannot reduce the length of the game.)  We will prove this claim through induction on $r$.  When $r=1$ we have $1(1+1)/2 = 1$, and it is clear that Burner can burn a path on one vertex in one round.  When $r=2$ we have $2(2+1)/2 = 3$; within two rounds, Burner can burn a path on two vertices by playing either vertex in round 1, and they can burn a path on three vertices by playing the central vertex in round 1.  Fix $r \ge 3$ and assume that within $r-2$ rounds, Burner can burn any path on $(r-2)(r-1)/2$ or fewer vertices.  Let the vertices of $P_n$ be $v_1, v_2, \dots, v_n$, in order.  If $n < r$, then no matter which vertex Burner burns in round 1, all $n$ vertices will be burned by the end of round $r$.  Otherwise, in round 1, Burner burns vertex $v_r$.  Note that by the end of round $r$, the fire on vertex $v_r$ will have spread to vertices $v_1, \dots, v_{2r-1}$.  Hence, if $n \le 2r-1$, then all $n$ vertices will be burned by the end of round $n$.  Otherwise, Burner can burn the subpath consisting of vertices $v_{2r}, \dots, v_n$ within rounds $3, \dots, r$ by the induction hypothesis, since $n-(2r-1) \le r(r+1)/2 - 2r + 1 = (r^2-3r-2)/2 = (r-2)(r-1)/2$.  This establishes the claim that Burner can burn all of $P_n$ whenever $n \le r(r+1)/2$; the desired upper bound on $\bg(P_n)$ now follows because
\[n \le \frac{r(r+1)}{2} \quad \text{ implies } \quad 2n \le r^2+r = \left (r+\frac{1}{2}\right )^2-\frac{1}{4}, \quad \text{ so } \quad \sqrt{2n+\frac{1}{4}} \le r+\frac{1}{2},\]
so $\ceil{\sqrt{2n+\frac{1}{4}} - \frac{1}{2}}$ rounds suffice to burn $P_n$ in the Burner-start game.  

Similar arguments yield the claimed upper bound on $\bg'(P_n)$.  As in the preceding paragraph, it suffices to show that in the Staller-start game on $P_n$, Burner can ensure that all vertices of $P_n$ burn within $r$ rounds whenever $n \le r(r+1)/2$.  Suppose Staller plays vertex $v_i$ in the first round of the game; we consider two cases.
\begin{itemize}
\item \textbf{Case 1: $i \le r$ or $i \ge n-r+1$.}  By symmetry, we may suppose $i \le r$.  Staller's first move will ensure that vertices $v_1, \dots, v_r$ all burn within $r$ rounds.  Hence, Burner need only worry about burning $v_{r+1}, \dots, v_n$ in the remaining $r-1$ rounds; the arguments used for the Burner-start game show that Burner can do this provided that $n-r \le (r-1)r/2$, i.e. $n \le r(r+1)/2$.

\item \textbf{Case 2: $r+1 \le i \le n-r$.}  In this case, Staller's first move will ensure that vertices $v_{i-r+1}, \dots, v_{i+r-1}$ all burn within $r$ rounds, so Burner need only burn the subpath induced by $v_1, \dots, v_{i-r}$ and the subpath induced by $v_{i+r},\dots, v_n$ in the remaining $r-1$ rounds.  

We claim that in the Burner-start game, Burner can burn the graph $P_s \cup P_t$ within $r-1$ rounds provided that $s+t \le (r-2)(r-1)/2$.  We prove this through induction on $r$.  If $r \le 2$ then the claim is trivial, so suppose $r \ge 3$, and suppose without loss of generality that $s \ge t$.  
\begin{itemize}
\item If $s \le 2r-3$, then with his first move, Burner plays a central vertex of $P_s$.  This causes the entire $P_s$ to burn within $r-1$ rounds.  To finish the game, Burner need only burn $P_t$ within the final $r-3$ rounds; the arguments used above for the Burner-start game show that this is possible provided that $t \le (r-3)(r-2)/2$.  Since $s \le t$, we have 
\[t \le \floor{(s+t)/2} \le \floor{(r-1)(r-2)/4} \le (r-3)(r-2)/2,\]
so Burner can burn $P_s \cup P_t$ within $r-1$ rounds, as claimed. 
\item If instead $s > 2r-3$, then Burner plays the $(r-1)$st vertex of the $P_s$.  In the course of $r-1$ rounds, this move will cause the first $2r-3$ vertices of the $P_s$ to burn, so Burner need only burn $P_{s-2r+3}\cup P_t$ in the final $r-3$ rounds of the game.  By the induction hypothesis, this is possible provided that $s-2r+3+t \le (r-4)(r-3)/2$; this is equivalent to $s+t \le (r^2-7r+12)/2 + 2r-3 = (r^2-3r+6)/2$, which holds by the initial assumption that $s+t \le (r-2)(r-1)/2 = (r^2-3r+2)/2$.
\end{itemize}

It follows that Burner can burn the subpaths of $P_n$ induced by $v_1, \dots, v_{i-r}$ and by $v_{i+r},\dots, v_n$ within $r-1$ rounds provided that $(i-r)+(n-(i+r)+1) \le (r-2)(r-1)/2$, i.e. $n-2r+1 \le (r-2)(r-1)/2$, which is equivalent to the initial assumption that $n \le r(r+1)/2$.
\end{itemize}

Next, we turn to the lower bounds.  For the lower bound on $\bg(P_n)$, we claim that for all $k \ge 1$, Staller can ensure that:
\begin{itemize}
\item By the end of round $2k-1$, at most $2k^2-1$ vertices are burned, and the subgraph induced by the burned vertices contains at most $k$ components; and
\item By the end of round $2k$, at most $2k^2+2k$ vertices are burned, and the subgraph induced by the burned vertices contains at most $k$ components.
\end{itemize}
We proceed through induction on $k$.  Let $v$ be the vertex burned by Burner in round 1.  At the end of round 1, only one vertex -- namely, $v$ -- will have been burned, and there will be only one component of burned vertices.  In the spreading phase of round 2, the neighbors of $v$ will themselves burn, so in total at most three vertices will have been burned.  After this happens, assuming that the game has not yet finished, the graph will contain both burned and unburned vertices; Staller will burn any unburned vertex that is adjacent to a burned vertex.  Note that this does not increase the number of components of burned vertices.  Thus, at the end of round 2, at most four vertices will have been burned, and there will be at most one component of burned vertices.  Hence, the claim holds when $k=1$.

Suppose next that at the end of round $2k$, at most $2k^2+2k$ vertices have been burned and there are at most $k$ components of burned vertices.  In the spreading phase of round $2k+1$, every unburned vertex adjacent to a burned vertex becomes burned; since there are at most $k$ components of burned vertices, at most $2k$ additional vertices get burned, and the number of components of burned vertices cannot increase.  Next, Burner burns an additional vertex; in addition, the number of components of burned vertices could increase by 1.  Thus, at the end of the round, the number of burned vertices is at most $2k^2+2k+2k+1 = 2k^2+4k+1 = 2(k+1)^2-1$ and there are at most $k+1$ components of burned vertices, as claimed.

Similarly, in the spreading phase of round $2k+2$, at most $2(k+1)$ additional vertices burn, after which Staller burns any unburned vertex with a burned neighbor (assuming that the game has not yet finished).  Thus at the end of the round, the number of burned vertices is at most $2(k+1)^2-1+2(k+1)+1 = 2(k+1)^2+2(k+1)$ and there are at most $k+1$ components of burned vertices, as claimed.  Hence, the claim holds for all $k$.

Now suppose that the game finishes after $r$ rounds.  Since all $n$ vertices have been burned, Staller's strategy ensures that $n \le 2k^2-1$ if $r=2k-1$ for some $k$, and $n \le 2k^2+2k$ if $r=2k$.  In the former case, we have $k \ge \sqrt{(n+1)/2}$, hence $r = 2k-1 \ge \sqrt{2n+2}-1$; in the latter case we have $n \le 2((k+1/2)^2-1/4)$, hence $k \ge \sqrt{n/2+1/4}-\frac{1}{2}$, so $r = 2k \ge \sqrt{2n+1}-1$.  In either case $r \ge \sqrt{2n+1}-1$; since $r$ must be an integer, we have $r \ge \ceil{\sqrt{2n+1}-1}$, as claimed.

Similar arguments yield the claimed lower bound on $\bg'(P_n)$.  For the sake of brevity, we omit the details and only explain the changes needed to adapt the argument above to the Staller-start game.  The Staller-start game is a bit worse for Staller since she must create a new component of burned vertices in round 1.  Thus, in the Staller-start game, for all $k \ge 1$, Staller can only ensure that:
\begin{itemize}
\item By the end of round $2k-1$, at most $2k^2-1$ vertices are burned, and the subgraph induced by the burned vertices contains at most $k$ components, at least one of which includes an endpoint of $P_n$; and
\item By the end of round $2k$, at most $2k^2+2k-1$ vertices are burned, and the subgraph induced by the burned vertices contains at most $k+1$ components, at least one of which includes an endpoint of $P_n$.
\end{itemize}
(Note that the stipulation that one component of burned vertices contains an endpoint of $P_n$ arises from Staller playing her first move on $v_1$; note also that this component can only spread fire to one additional vertex -- not two -- in each round.)  If the game finishes after $r$ rounds, then we must have $n \le 2k^2-1$ if $r = 2k-1$ for some $k$ and $n \le 2k^2+2k-1$ if $r=2k$ for some $k$.  The former case implies $r \ge \sqrt{2n+2}-1$; the latter case implies $n \le 2((k+1/2)^2-3/4)$, so $r = 2k \ge 2(\sqrt{n/2+3/4}-1/2) = \sqrt{2n+3}-1$.  In either case, $r \ge \ceil{\sqrt{2n+2}-1}$, as claimed. 
\end{proof}

The strategies for Burner and for Staller in the proof of Theorem \ref{thm:paths} work on $C_n$ as well as on $P_n$. Again, the upper and lower bounds given differ by at most 1.

\begin{theorem}\label{thm:cycles}
For all $n \ge 3$, we have 
\[\ceil{\sqrt{2n+1}-1} \le \bg(C_n) \le \ceil{\sqrt{2n+\frac{1}{4}}-\frac{1}{2}}\]
and
\[\ceil{\sqrt{2n+7}-2} \le \bg'(C_n) \le \ceil{\sqrt{2n+\frac{17}{4}}-\frac{3}{2}}.\]
\end{theorem}
\begin{proof}
For the Burner-start game, the claimed bounds follow by the same arguments used to bound $\bg(P_n)$ in Theorem \ref{thm:paths}.  For the Staller-start game, the arguments are very similar to those used in Theorem \ref{thm:paths}, with the following changes:
\begin{itemize}
\item For the upper bound, we may assume by symmetry that Staller's first move is on vertex $v_{r}$, which will cause vertices $v_1, \dots, v_{2r-1}$ to burn within $r$ rounds; Burner thus need only worry about burning the subpath induced by $v_{2r}, \dots, v_n$.
\item For the lower bound, unlike on $P_n$, Staller is no longer able to ensure that the component of burned vertices she creates in round 1 can only spread fire in one direction, rather than two.
\end{itemize}
We leave verification of the details to the reader.
\end{proof}

Although the lower and upper bounds in both Theorem \ref{thm:paths} and Theorem \ref{thm:cycles} differ only by 1, computer search suggests that neither bound consistently yields the true value of the game burning number.  We suspect that determining the exact values of $\bg(P_n)$ and $\bg(C_n)$ may be quite complicated.

Theorem \ref{thm:paths} shows that $\bg(P_n) = (1+o(1))\sqrt{2n}$.  The burning number conjecture posits that paths have the maximum burning number among connected $n$-vertex graphs, so it is natural to hypothesize that perhaps the same is true in the burning game.  As we next show, this is in fact true, at least in an asymptotic sense: for every connected $n$-vertex graph $G$, we have $\bg(G) \le (1+o(1))\sqrt{2n}$.  We obtain this result as a consequence of recent progress toward the burning number conjecture.  We begin with a lemma about the burning process.  (Note that this lemma deals with the original burning \textit{process}, not the burning \textit{game}.)  

\begin{lemma}\label{lem:burning_conj_helper}
If every connected graph $H$ satisfies $b(H) \le f(\size{V(H)})$, then for every connected graph $G$ we have $b(G^2) \le f(k)+1$, where $k$ denotes the smallest size of a partite set in any spanning tree of $G$.
\end{lemma}
\begin{proof}
Let $G$ be a connected graph, and let $T$ be a spanning tree of $G$ with one partite set having size $k$.  Since $T$ is a spanning subgraph of $G$, it follows that $T^2$ is a spanning subgraph of $G^2$ and hence that $b(G^2) \le b(T^2)$, so it suffices to argue that $b(T^2) \le f(k)+1$.

Let $X$ be a partite set of $T$ with $\size{X} = k$ and let $H$ be the subgraph of $T^2$ induced by $X$.  Note that connectivity of $T$ implies connectivity of $H$.  Moreover, if we have a strategy to burn $H$, then playing the same moves in $T^2$ and waiting one additional turn will ensure that we burn all vertices in $T^2$.  More precisely, let $v_1, v_2, \dots, v_m$ be a burning sequence in $H$; we claim that for any vertex $w$ in $T^2$ (other than $v_1, \dots, v_m$), the sequence $v_1, v_2, \dots, v_m, w$ is a burning sequence in $T^2$.  Two vertices of $X$ are adjacent in $H$ if and only if they are adjacent in $T^2$; hence, for any vertices $x,y \in X$, we have $\dist_{T^2}(x,y) = \dist_H(x,y)$.  Since $v_1, v_2, \dots, v_m$ is a burning sequence in $H$, we have
\[X \subseteq N_{m-1}[v_1] \cup N_{m-2}[v_2] \cup \dots \cup N_0[v_m]\]
(in both $H$ and $T^2$).  By connectedness and bipartiteness of $T$, every vertex of $T$ either belongs to $X$ or is adjacent to some vertex of $X$ in $T$ (and hence also in $T^2$); thus,
\[V(T^2) \subseteq N[X] \subseteq N_m[v_1] \cup N_{m-1}[v_2] \cup \dots \cup N_1[v_m] \subseteq N_m[v_1] \cup N_{m-1}[v_2] \cup \dots \cup N_1[v_m] \cup N_0[w],\]
so $v_1, v_2, \dots, v_m, w$ is a burning sequence in $T^2$. Thus, 
\[b(T^2) \le b(H) + 1 \le f(\size{V(H)}) + 1 = f(k)+1,\]
as claimed.
\end{proof}

Norin and Turcotte \cite{NT24} recently showed that $b(G) \leq (1+o(1))\sqrt{n}$ for all connected $n$-vertex graphs; this result, together with Lemma \ref{lem:burning_conj_helper}, yields the following general upper bound on $\bg(G)$.

\begin{corollary}\label{cor:burning_conj}
If $G$ is a connected graph on $n$ vertices, then
\[\bg(G) \le (1+o(1))\sqrt{2n}.\]
\end{corollary}
\begin{proof}
Let $G$ be a connected $n$-vertex graph.  In any spanning tree of $G$, some partite set has size at most $n/2$; thus Proposition \ref{prop:trivial-bounds}, the Tree Reduction Theorem, Lemma \ref{lem:burning_conj_helper}, and the aforementioned result of Norin and Turcotte together yield
\[\bg(G) \le 2b(G^2)-1 \le 2((1+o(1))\sqrt{n/2}+1)-1 \le 2(1+o(1))\sqrt{n/2} +1 = (1+o(1))\sqrt{2n},\]
as claimed.
\end{proof}

Note that Corollary \ref{cor:burning_conj} and Theorem \ref{thm:paths} show that for every connected $n$-vertex graph $G$, we have $\bg(G) \leq (1+o(1))\bg(P_n)$.  Hence paths are, in an asymptotic sense, graphs with largest possible game burning number.

If the burning number conjecture is in fact true, then an argument similar to that used in Corollary \ref{cor:burning_conj} yields a tighter upper bound on $\bg(G)$.

\begin{corollary}\label{cor:burning_conj_2}
If the burning number conjecture is true, then for every connected graph $G$ on $n$ vertices, we have $\bg(G) \le \floor{\sqrt{2n}}+3$.
\end{corollary}
\begin{proof}
As in the proof of Corollary \ref{cor:burning_conj}, in any spanning tree of $G$, some partite set has size at most $n/2$. Thus, assuming the truth of the burning number conjecture, we have
\[\bg(G) \le 2b(G^2)-1 \le 2(\ceil{\sqrt{n/2}}+1)-1 \le 2(\sqrt{n/2}+2)-1 = 2\sqrt{n/2}+3 = \sqrt{2n}+3;\]
since $\bg(G)$ must be an integer, it follows that $\bg(G) \le \floor{\sqrt{2n}}+3$, as claimed.
\end{proof}

\section{Computational Complexity}
\label{sec:complexity}

We conclude the paper by examining the computational complexity of the burning game.  In particular, we consider the following decision problem:\\

\noindent \textsc{BurningGame}: Given a graph $G$ and positive integer $k$, is $\bg(G)$ at most $k$?\\

As mentioned in the introduction, determining the burning number of graphs is NP-complete even for trees with maximum degree 3 and several other classes of graphs \cite{bessy2017burning}.  Consequently, it is natural to suspect that \textsc{BurningGame} is NP-hard (if not NP-complete).  We next prove that this is in fact the case.  

We will prove the NP-hardness of this problem through reduction from the problem 3-SAT, which is well-known to be NP-complete (see, for example, \cite{Karp1972}).  A 3-SAT instance consists of a propositional logic formula $\varphi$ in conjunctive normal form, in which each clause consists of three literals (where a \textit{literal} is either a variable or the negation of a variable).  That is,
\[\varphi = C_1 \wedge C_2 \wedge \dots \wedge C_m,\]
where each $C_i$ has the form $a \vee b \vee c$, with $a$, $b$, and $c$ being literals.  The \textsc{3-SAT} problem is to determine whether or not $\varphi$ has a \textit{satisfying assignment}, i.e. whether or not it is possible to choose truth values for the variables under which $\varphi$ is true.

\begin{theorem}\label{thm:NP_hard}
\textsc{BurningGame} is NP-hard.
\end{theorem}
\begin{proof}
As noted above, we will prove that \textsc{BurningGame} is NP-hard through reduction from \textsc{3-SAT}.  Let $\varphi$ be an instance of \textsc{3-SAT}, and suppose that $\varphi$ uses variables $x_1, x_2, \dots, x_n$; we show how to construct a graph $G$ such that $\varphi$ has a satisfying assignment if and only if $\bg(G) \le 4n+1$.

We construct such a graph $G$ as follows. 
First, for each variable $x_i$, we construct a corresponding subgraph $G_i$, as follows:
\begin{itemize}
\item Add vertices $x_i^T$ and $x_i^F$, as well as the edge $x_i^Tx_i^F$.
\item Add vertices $a_i^1, a_i^2, \dots, a_i^{4n-2i+1}$, edges $a_i^1a_i^2, a_i^2a_i^3, \dots, a_i^{4n-2i}a_i^{4n-2i+1}$, and edges from $a_i^1$ to both $x_i^T$ and $x_i^F$.  
\item Add vertices $t_i^1, t_i^2, \dots, t_i^{4n-2i}$, edges $t_i^1t_i^2, t_i^2t_i^3, \dots, t_i^{4n-2i-1}t_i^{4n-2i}$, and edge $t_i^1x_i^T$; we refer to this path as the \textit{true fuse} for $x_i$.
\item Add vertices $f_i^1, f_i^2, \dots, f_i^{4n-2i}$, edges $f_i^1f_i^2, f_i^2f_i^3, \dots, f_i^{4n-2i-1}f_i^{4n-2i}$, and edge $f_i^1x_i^F$; we refer to this path as the \textit{false fuse} for $x_i$.
\item Add a vertex $z_i$ and edges $z_it_i^{4n-2i}$ and $z_if_i^{4n-2i}$.
\end{itemize}

Now, for each clause $C$ of $\varphi$, we add two \textit{clause vertices} $v_C$ and $v'_C$.  For each $i \in\{1,\dots, n\}$, add edges joining $t_i^{4n-2i}$ to all clause vertices $v_C$ and $v'_C$ such that $C$ contains the literal $x_i$; likewise, add edges joining $f_i^{4n-2i}$ to each clause vertex $v_C$ or $v'_C$ such that $C$ contains the literal $\overline{x_i}$.  We additionally add components $H_1, H_2, \dots, H_n$, where $H_i$ is a path on $4(n-i)+3$ vertices.  We refer to each vertex added to $G$ so far as a \textit{base vertex}; furthermore, for each base vertex $v$, we add a pendant vertex adjacent to $v$.
(Note that we do not consider the pendant vertices themselves to belong to the $G_i$ or $H_i$.)  Finally, we add one isolated vertex, $z$. An example is shown on Figure~\ref{fig:NP_hard}.

\begin{figure}
\centering
\begin{tikzpicture} [inner sep=0mm, thick,
 dot/.style={draw=black, fill=black, circle, minimum size=0.010cm},
 smallvertex/.style={draw=black, circle, minimum size=0.011cm},
 vertex/.style={draw=black, fill=black, circle, minimum size=0.2cm},
 xscale=1,yscale=1.5]

\node (vc1) at (1.6,-7.0) [vertex] {};
\node at (vc1) [below = 0.3cm] {$v_{C_1}$};
\node (vc1p) at ($(vc1)+(0.8,0)$) [vertex] {};
\node at (vc1p) [below = 0.2cm] {$v'_{C_1}$};

\node (vc2) at (5.6,-7.0) [vertex] {};
\node at (vc2) [below = 0.3cm] {$v_{C_2}$};
\node (vc2p) at ($(vc2)+(0.8,0)$) [vertex] {};
\node at (vc2p) [below = 0.2cm] {$v'_{C_2}$};

\node (spacer) at (0,-8) [] {};

% n = 3 here...
\node (a11) at (0, 0.4) [vertex] {};
\node at (a11) [right = 0.2cm] {$a_1^{11}$};
\node (a12) at ($(a11)+(0,-0.5)$) [vertex] {};
\node at (a12) [right = 0.2cm] {$a_1^{10}$};
\node (d1a1) at ($(a12)+(0,-0.35)$) [dot] {};
\node (d1a2) at ($(a12)+(0,-0.45)$) [dot] {};
\node (d1a3) at ($(a12)+(0,-0.55)$) [dot] {};
\node (d1a4) at ($(a12)+(0,-0.65)$) [dot] {};
\node (d1a5) at ($(a12)+(0,-0.75)$) [dot] {};
\node (d1a6) at ($(a12)+(0,-0.85)$) [dot] {};
\node (d1a7) at ($(a12)+(0,-0.95)$) [dot] {};
\node (a13) at ($(a12)+(0,-1.3)$) [vertex] {};
\node at (a13) [right = 0.2cm] {$a_1^{2}$};
\node (a14) at ($(a13)+(0,-0.5)$) [vertex] {};
\node at (a14) [right = 0.2cm] {$a_1^{1}$};

\node (x1t) at ($(a14)+(-0.6,-0.45)$) [vertex] {};
\node at (x1t) [left = 0.2cm] {$x_1^{T}$};
\node (t11) at ($(x1t)+(0,-0.5)$) [vertex] {};
\node at (t11) [left = 0.2cm] {$t_1^1$};
\node (t12) at ($(t11)+(0,-0.5)$) [vertex] {};
\node at (t12) [left = 0.2cm] {$t_1^2$};
\node (d1t1) at ($(t12)+(0,-0.35)$) [dot] {};
\node (d1t2) at ($(t12)+(0,-0.45)$) [dot] {};
\node (d1t3) at ($(t12)+(0,-0.55)$) [dot] {};
\node (d1t4) at ($(t12)+(0,-0.65)$) [dot] {};
\node (d1t5) at ($(t12)+(0,-0.75)$) [dot] {};
\node (d1t6) at ($(t12)+(0,-0.85)$) [dot] {};
\node (d1t7) at ($(t12)+(0,-0.95)$) [dot] {};
\node (t13) at ($(t12)+(0,-1.3)$) [vertex] {};
\node at (t13) [left = 0.2cm] {$t_1^{9}$};
\node (t14) at ($(t13)+(0,-0.5)$) [vertex] {};
\node at (t14) [left = 0.2cm] {$t_1^{10}$};

\node (x1f) at ($(a14)+(0.6,-0.45)$) [vertex] {};
\node at (x1f) [right = 0.2cm] {$x_1^{F}$};
\node (f11) at ($(x1f)+(0,-0.5)$) [vertex] {};
\node at (f11) [right = 0.2cm] {$f_1^1$};
\node (f12) at ($(f11)+(0,-0.5)$) [vertex] {};
\node at (f12) [right = 0.2cm] {$f_1^2$};
\node (d1f1) at ($(f12)+(0,-0.35)$) [dot] {};
\node (d1f2) at ($(f12)+(0,-0.45)$) [dot] {};
\node (d1f3) at ($(f12)+(0,-0.55)$) [dot] {};
\node (d1f4) at ($(f12)+(0,-0.65)$) [dot] {};
\node (d1f5) at ($(f12)+(0,-0.75)$) [dot] {};
\node (d1f6) at ($(f12)+(0,-0.85)$) [dot] {};
\node (d1f7) at ($(f12)+(0,-0.95)$) [dot] {};
\node (f13) at ($(f12)+(0,-1.3)$) [vertex] {};
\node at (f13) [right = 0.2cm] {$f_1^{9}$};
\node (f14) at ($(f13)+(0,-0.5)$) [vertex] {};
\node at (f14) [right = 0.2cm] {$f_1^{10}$};
\node (z1) at ($(f14)+(-0.6,0.25)$) [vertex] {};
\node at (z1) [below = 0.2cm] {$z_1$};

\draw (a11) -- (a12) -- ($(a12)+(0,-0.25)$);
\draw ($(a13)+(0,0.25)$) -- (a14) -- (x1t) -- (x1f) -- (a14);
\draw (x1t) -- (t11) -- (t12) -- ($(t12)+(0,-0.25)$);
\draw ($(t13)+(0,0.25)$) -- (t13) -- (t14);
\draw (x1f) -- (f11) -- (f12) -- ($(f12)+(0,-0.25)$);
\draw ($(f13)+(0,0.25)$) -- (f13) -- (f14);
\draw (t14) -- (z1) -- (f14);

% Start of G2
\node (a21) at (4, 0.2) [vertex] {};
\node at (a21) [right = 0.2cm] {$a_2^{9}$};
\node (a22) at ($(a21)+(0,-0.5)$) [vertex] {};
\node at (a22) [right = 0.2cm] {$a_2^{8}$};
\node (d2a1) at ($(a22)+(0,-0.35)$) [dot] {};
\node (d2a2) at ($(a22)+(0,-0.45)$) [dot] {};
\node (d2a3) at ($(a22)+(0,-0.55)$) [dot] {};
\node (d2a4) at ($(a22)+(0,-0.65)$) [dot] {};
\node (d2a5) at ($(a22)+(0,-0.75)$) [dot] {};
\node (a23) at ($(a22)+(0,-1.1)$) [vertex] {};
\node at (a23) [right = 0.2cm] {$a_2^{2}$};
\node (a24) at ($(a23)+(0,-0.5)$) [vertex] {};
\node at (a24) [right = 0.2cm] {$a_2^{1}$};

\node (x2t) at ($(a24)+(-0.6,-0.45)$) [vertex] {};
\node at (x2t) [left = 0.2cm] {$x_2^{T}$};
\node (t21) at ($(x2t)+(0,-0.5)$) [vertex] {};
\node at (t21) [left = 0.2cm] {$t_2^1$};
\node (t22) at ($(t21)+(0,-0.5)$) [vertex] {};
\node at (t22) [left = 0.2cm] {$t_2^2$};
\node (d2t1) at ($(t22)+(0,-0.35)$) [dot] {};
\node (d2t2) at ($(t22)+(0,-0.45)$) [dot] {};
\node (d2t3) at ($(t22)+(0,-0.55)$) [dot] {};
\node (d2t4) at ($(t22)+(0,-0.65)$) [dot] {};
\node (d2t5) at ($(t22)+(0,-0.75)$) [dot] {};
\node (t23) at ($(t22)+(0,-1.1)$) [vertex] {};
\node at (t23) [left = 0.2cm] {$t_2^{7}$};
\node (t24) at ($(t23)+(0,-0.5)$) [vertex] {};
\node at (t24) [left = 0.2cm] {$t_2^{8}$};

\node (x2f) at ($(a24)+(0.6,-0.45)$) [vertex] {};
\node at (x2f) [right = 0.2cm] {$x_2^{F}$};
\node (f21) at ($(x2f)+(0,-0.5)$) [vertex] {};
\node at (f21) [right = 0.2cm] {$f_2^1$};
\node (f22) at ($(f21)+(0,-0.5)$) [vertex] {};
\node at (f22) [right = 0.2cm] {$f_2^2$};
\node (d2f1) at ($(f22)+(0,-0.35)$) [dot] {};
\node (d2f2) at ($(f22)+(0,-0.45)$) [dot] {};
\node (d2f3) at ($(f22)+(0,-0.55)$) [dot] {};
\node (d2f4) at ($(f22)+(0,-0.65)$) [dot] {};
\node (d2f5) at ($(f22)+(0,-0.75)$) [dot] {};
\node (f23) at ($(f22)+(0,-1.1)$) [vertex] {};
\node at (f23) [right = 0.2cm] {$f_2^{7}$};
\node (f24) at ($(f23)+(0,-0.5)$) [vertex] {};
\node at (f24) [right = 0.2cm] {$f_2^{8}$};
\node (z2) at ($(f24)+(-0.6,0.25)$) [vertex] {};
\node at (z2) [below = 0.2cm] {$z_2$};

\draw (a21) -- (a22) -- ($(a22)+(0,-0.25)$);
\draw ($(a23)+(0,0.25)$) -- (a24) -- (x2t) -- (x2f) -- (a24);
\draw (x2t) -- (t21) -- (t22) -- ($(t22)+(0,-0.25)$);
\draw ($(t23)+(0,0.25)$) -- (t23) -- (t24);
\draw (x2f) -- (f21) -- (f22) -- ($(f22)+(0,-0.25)$);
\draw ($(f23)+(0,0.25)$) -- (f23) -- (f24);
\draw (t24) -- (z2) -- (f24);

% Start of G3
\node (a31) at (8, 0) [vertex] {};
\node at (a31) [right = 0.2cm] {$a_3^{7}$};
\node (a32) at ($(a31)+(0,-0.5)$) [vertex] {};
\node at (a32) [right = 0.2cm] {$a_3^{6}$};
\node (d3a1) at ($(a32)+(0,-0.35)$) [dot] {};
\node (d3a2) at ($(a32)+(0,-0.45)$) [dot] {};
\node (d3a3) at ($(a32)+(0,-0.55)$) [dot] {};
\node (a33) at ($(a32)+(0,-0.9)$) [vertex] {};
\node at (a33) [right = 0.2cm] {$a_3^{2}$};
\node (a34) at ($(a33)+(0,-0.5)$) [vertex] {};
\node at (a34) [right = 0.2cm] {$a_3^{1}$};

\node (x3t) at ($(a34)+(-0.6,-0.45)$) [vertex] {};
\node at (x3t) [left = 0.2cm] {$x_3^{T}$};
\node (t31) at ($(x3t)+(0,-0.5)$) [vertex] {};
\node at (t31) [left = 0.2cm] {$t_3^1$};
\node (t32) at ($(t31)+(0,-0.5)$) [vertex] {};
\node at (t32) [left = 0.2cm] {$t_3^2$};
\node (d3t1) at ($(t32)+(0,-0.35)$) [dot] {};
\node (d3t2) at ($(t32)+(0,-0.45)$) [dot] {};
\node (d3t3) at ($(t32)+(0,-0.55)$) [dot] {};
\node (t33) at ($(t32)+(0,-0.9)$) [vertex] {};
\node at (t33) [left = 0.2cm] {$t_3^{5}$};
\node (t34) at ($(t33)+(0,-0.5)$) [vertex] {};
\node at (t34) [left = 0.2cm] {$t_3^{6}$};

\node (x3f) at ($(a34)+(0.6,-0.45)$) [vertex] {};
\node at (x3f) [right = 0.2cm] {$x_3^{F}$};
\node (f31) at ($(x3f)+(0,-0.5)$) [vertex] {};
\node at (f31) [right = 0.2cm] {$f_3^1$};
\node (f32) at ($(f31)+(0,-0.5)$) [vertex] {};
\node at (f32) [right = 0.2cm] {$f_3^2$};
\node (d3f1) at ($(f32)+(0,-0.35)$) [dot] {};
\node (d3f2) at ($(f32)+(0,-0.45)$) [dot] {};
\node (d3f3) at ($(f32)+(0,-0.55)$) [dot] {};
\node (f33) at ($(f32)+(0,-0.9)$) [vertex] {};
\node at (f33) [right = 0.2cm] {$f_3^{5}$};
\node (f34) at ($(f33)+(0,-0.5)$) [vertex] {};
\node at (f34) [right = 0.2cm] {$f_3^{6}$};
\node (z3) at ($(f34)+(-0.6,0.25)$) [vertex] {};
\node at (z3) [below = 0.2cm] {$z_3$};

\draw (a31) -- (a32) -- ($(a32)+(0,-0.25)$);
\draw ($(a33)+(0,0.25)$) -- (a34) -- (x3t) -- (x3f) -- (a34);
\draw (x3t) -- (t31) -- (t32) -- ($(t32)+(0,-0.25)$);
\draw ($(t33)+(0,0.25)$) -- (t33) -- (t34);
\draw (x3f) -- (f31) -- (f32) -- ($(f32)+(0,-0.25)$);
\draw ($(f33)+(0,0.25)$) -- (f33) -- (f34);
\draw (t34) -- (z3) -- (f34);

% H1
\node (h11) at (11,0.1) [vertex] {};
\node (h12) at ($(h11)+(0,-0.5)$) [vertex] {};
\node (h13) at ($(h12)+(0,-0.5)$) [vertex] {};
\node (h14) at ($(h13)+(0,-0.5)$) [vertex] {};
\node (h15) at ($(h14)+(0,-0.5)$) [vertex] {};
\node (h16) at ($(h15)+(0,-0.5)$) [vertex] {};
\node (h17) at ($(h16)+(0,-0.5)$) [vertex] {};
\node (h18) at ($(h17)+(0,-0.5)$) [vertex] {};
\node (h19) at ($(h18)+(0,-0.5)$) [vertex] {};
\node (h110) at ($(h19)+(0,-0.5)$) [vertex] {};
\node (h111) at ($(h110)+(0,-0.5)$) [vertex] {};

\draw (h11) -- (h12) -- (h13) -- (h14) -- (h15) -- (h16) -- (h17) -- (h18) -- (h19) -- (h110) -- (h111);

% H2
\node (h21) at (13,0.2) [vertex] {};
\node (h22) at ($(h21)+(0,-0.5)$) [vertex] {};
\node (h23) at ($(h22)+(0,-0.5)$) [vertex] {};
\node (h24) at ($(h23)+(0,-0.5)$) [vertex] {};
\node (h25) at ($(h24)+(0,-0.5)$) [vertex] {};
\node (h26) at ($(h25)+(0,-0.5)$) [vertex] {};
\node (h27) at ($(h26)+(0,-0.5)$) [vertex] {};

\draw (h21) -- (h22) -- (h23) -- (h24) -- (h25) -- (h26) -- (h27);

% H3
\node (h31) at (13,-4.0) [vertex] {};
\node (h32) at ($(h31)+(0,-0.5)$) [vertex] {};
\node (h33) at ($(h32)+(0,-0.5)$) [vertex] {};

\draw (h31) -- (h32) -- (h33);

\node (o) at (12, -6.2) [vertex] {};
\node at (o) [below = 3mm] {$z$};

% Clause vertex edges
\draw (t14) -- (vc1);  \draw (t14) -- (vc1p);
\draw (f24) -- (vc1);  \draw (f24) -- (vc1p);
\draw (f34) -- (vc1);  \draw (f34) -- (vc1p);

\draw (t14) -- (vc2);  \draw (t14) -- (vc2p);
\draw (t24) -- (vc2);  \draw (t24) -- (vc2p);
\draw (f34) -- (vc2);  \draw (f34) -- (vc2p);

\draw [rounded corners=3mm, dashed] ($(t14)+(-0.9,-0.5)$) rectangle ($(a11)+(1.5,0.5)$);
\draw [rounded corners=3mm, dashed] ($(t24)+(-0.9,-0.5)$) rectangle ($(a21)+(1.5,0.5)$);
\draw [rounded corners=3mm, dashed] ($(t34)+(-0.9,-0.5)$) rectangle ($(a31)+(1.5,0.5)$);

\node at (a11) [above=11mm] {$G_1$}; 
\node at (a21) [above=11mm] {$G_2$}; 
\node at (a31) [above=11mm] {$G_3$}; 

\draw [rounded corners=3mm, dashed] ($(h111)+(-0.3,-0.25)$) rectangle ($(h11)+(0.3,0.25)$);
\draw [rounded corners=3mm, dashed] ($(h27)+(-0.3,-0.25)$) rectangle ($(h21)+(0.3,0.25)$);
\draw [rounded corners=3mm, dashed] ($(h33)+(-0.3,-0.25)$) rectangle ($(h31)+(0.3,0.25)$);

\node at (h11) [above=5mm] {$H_1$}; 
\node at (h21) [above=5mm] {$H_2$}; 
\node at (h31) [above=5mm] {$H_3$};

\end{tikzpicture}

\caption{The graph $G$ corresponding to the propositional logic formula $C_1 \wedge C_2$, where $C_1 = x_1 \vee \overline{x_2} \vee \overline{x_3}$ and $C_2 = x_1 \vee x_2 \vee \overline{x_3}$ (pendant vertices not pictured).}
\label{fig:NP_hard}
\end{figure}

We claim that $\varphi$ has a satisfying assignment if and only if $\bg(G) \le 4n+1$.  Intuitively, $G$ has been designed so that under optimal play, Burner must play either $x_i^T$ or $x_i^F$ on his $i$th turn for $i \in \{1, \dots, n\}$, after which he must play one move in each of the components $H_1, \dots, H_{n}$, in order.  If Burner plays $x_i^T$ on his $i$th turn, then this is analogous to setting variable $x_i$ to be true in the 3-SAT instance; likewise, playing $x_i^F$ is analogous to setting $x_i$ to be false.  The lengths of the true fuses and false fuses have been chosen so that, provided Burner plays as claimed (and Staller plays optimally), no clause vertices burn until the spreading phase of round $4n$.  Note that if Burner plays $x_i^T$, then he does so in the $(2i-1)$th round of the game, so the true fuse for $x_i$ begins burning in round $2i-1$, while the false fuse does not begin burning until round $2i$. As such, in round $4n$, the fire from $x_i^T$ spreads to all clause vertices adjacent to $t_i^{4n-2i}$, but not to those adjacent to $f_i^{4n-2i}$.  Likewise, playing $x_i^F$ ensures that the clause vertices adjacent to $f_i^{4n-2i}$ burn in round $4n$, rather than those adjacent to $t_i^{4n-2i}$. Hence, choosing values for $x_1, \dots, x_n$ that satisfy $\varphi$ will ensure that every clause of $\phi$ contains a true literal, which in turn ensures that all clause vertices burn in round $4n$, and the attached pendant vertices burn in round $4n+1$.

We begin with some observations about optimal play by Burner.  First, we show that Burner never needs to play a pendant vertex, except perhaps in the final round of the game.  Suppose that in the $k$th round of the game, Burner plays a pendant vertex $v$ with neighbor $w$, and additionally suppose that the game does not end in round $k$.  We claim that Burner would do no worse by playing $w$ instead of $v$.  To see this, let $S$ denote the set of burned vertices prior to the selection phase of round $k$; additionally, let $S_v$ (resp. $S_w$) denote the set of burned vertices prior to the selection phase of round $k+1$, provided that Burner plays $v$ (resp. $w$) in round $k$.  We have $S_v = N[S \cup\{v\}] = N[S] \cup N[v]$ and $S_w = N[S \cup\{w\}] = N[S] \cup N[w]$; since $N[v] \subseteq N[w]$ we thus have $S_v \subseteq S_w$.  The Continuation Principle now implies that Burner does at least as well by playing $w$ as by playing $v$.

Next, we show that Burner never needs to play the last base vertex in any component of $G$, except perhaps on his final turn of the game. Suppose that in round $k$, Burner plays some base vertex $v$, which is the last unburned base vertex in component $K$.  Moreover, suppose that Staller responds by playing some vertex $w$ in round $k+1$, and suppose that the game continues until at least round $k+2$.  We claim that Burner would do no worse by playing $w$ instead of $v$.  Let $S$ denote the set of burned vertices after the spreading phase of round $k$, and let $S_v$ (resp. $S_w$) denote the set of burned vertices after the spreading phase of round $k+2$, provided that Burner plays $v$ (resp. $w$) in round $k$.  Note that every component of $G$ -- other than the isolated vertex $z$ -- contains at least two base vertices, so $S$ must contain some burned non-pendant neighbor of $v$ in $C$; by the end of the spreading phase of round $k+2$, the fire from this neighbor will have spread to both $v$ and its attached pendant vertex.  Hence $N_2[v] \subseteq N_2[S]$ and so
\[S_v = N[N[S \cup \{v\}] \cup \{w\}] = N_2[S] \cup N_2[v] \cup N[w] \subseteq N_2[S] \cup N[w].\]
On the other hand,
\[S_w = N[N[S \cup \{w\}] \cup \{x\}] = N_2[S] \cup N_2[w] \cup N[x] \supseteq N_2[S] \cup N_2[w] \supseteq N_2[S] \cup N[w] \supseteq S_v,\]
where $x$ denotes the vertex played by Staller in round $k+1$ provided that Burner plays $w$ in round $k$.  Since $S_w \supseteq S_v$, the Continuation Principle implies that Burner does at least as well by playing $w$ as by playing $v$, as claimed.  Similar reasoning shows that Burner never needs to play vertex $z$, except perhaps on his final turn.

Hence we may assume that Burner never plays a pendant vertex, nor does he ever play the last non-pendant vertex in any component of $G$ (except perhaps on his last turn of the game). Armed with these observations, we are ready to present an optimal strategy for Staller.  On each turn -- except, perhaps, for her final turn -- Staller simply plays any unburned pendant vertex whose neighbor is already burned.  To see that some such vertex must exist, consider Burner's preceding move.  Burner must have played a base vertex in some component $K$ of $G$ and, moreover, cannot have played the last remaining unburned base vertex in $K$.  Consequently, during the subsequent spreading phase, the fire must have spread to some previously-unburned base vertex $v$ in $K$.  Let $w$ denote the pendant vertex adjacent to $v$; it is clear that $w$ cannot yet have been burned, so $w$ is a vertex of the type Staller seeks.  It is relatively straightforward to see that this Staller strategy is optimal.  Since $w$ already has a burned neighbor, fire will spread to $w$ within one round, regardless of whether or not Staller selects it; hence, by burning $w$, Staller effectively avoids enlarging the set of burned vertices.  By the Continuation Principle, Staller cannot possibly hope to do any better than this.  Furthermore, note that if the game lasts for exactly $4n+1$ rounds, then Burner takes the last turn; it follows that even on Staller's final turn, she may play an unburned pendant vertex with burned neighbor.   

Burner's optimal strategy is more complicated.  We begin with an overview.  Assuming that Staller follows the optimal strategy presented above, Burner will first play a vertex in $G_1$, then one in $G_2$, and so on up through $G_n$; after this point, Burner will play in $H_1$, then $H_2$, and so on up through $H_n$.  Finally, Burner plays $z$.  (As we will see, the game will end either after Burner plays $z$ or after the subsequent spreading phase, so Burner will not have any further turns.) When playing in $G_i$, Burner will play either $x_i^T$ or $x_i^F$.  Burner will choose their moves according to a satisfying assignment for $\varphi$ (provided that one exists); playing $x_i^T$ will be tantamount to setting variable $x_i$ to be true, and playing $x_i^F$ will be tantamount to setting it to be false.  When playing in $H_i$, Burner will play a central vertex.

We first claim that if $\varphi$ has a satisfying assignment, then Burner's strategy ensures that the game ends within $4n+1$ rounds.  It suffices to argue that Burner's strategy will cause all base vertices of $G$ to burn by the end of round $4n$; consequently, all still-unburned pendant vertices would burn during the spreading phase of round $4n+1$, leaving only $z$ to be burned during the selection phase of round $4n+1$.  Burner will play in $G_i$ in round $2i-1$.  Both $x_i^T$ and $x_i^F$ are within distance $4n-2i+1$ of all base vertices of $G_i$.  Thus, whether Burner plays $x_i^T$ or $x_i^F$, all base vertices of $G_i$ will burn by the end of round $4n$.  Likewise, Burner will play in $H_i$ in round $2n+2i-1$; since $H_i$ has radius $2n-2i+1$, Burner's move in $H_i$ will ensure that all base vertices of $H_i$ burn by the end of round $4n$.  Finally, note that if Burner plays $x_i^T$ in round $2i-1$, then vertex $t_i^{4n-2i}$ will burn by the end of round $4n-1$, and as such, for all clauses $C$ of $\varphi$ that contain $x_i$, all clause vertices $v_C$ and $v'_C$ will burn by the end of round $4n$; likewise, if Burner plays $x_i^F$ in round $2i-1$, the for all clauses $C$ that contain $\overline{x_i}$, the clause vertices $v_C$ and $v'_C$ will burn by the end of round $4n$.  Consequently, if $\varphi$ has a satisfying assignment, then Burner's moves will cause all clause vertices to burn by the end of round $4n$.  Thus if $\varphi$ has a satisfying assignment, then all base vertices of $G$ burn by the end of round $4n$, hence $\bg(G) \le 4n+1$.

We next argue that if Burner does not or cannot follow this strategy -- for example, because there is no satisfying assignment for $\varphi$ -- then $G$ will not be fully burned by the end of round $4n+1$ (provided that Staller plays optimally).  We prove this through a series of claims.  In what follows, we assume that 
Staller follows the optimal strategy presented above.

We first claim that Burner must play at least one move in each $G_i$ and each $H_i$.  As argued above, with the possible exception of the last move of the game, Staller only burns pendant vertices with already-burned neighbors.  Hence, all base vertices in $G$ must be burned by Burner (either directly, or as a result of fire spreading from Burner's moves).  It follows immediately that Burner must play at least one move in each $H_i$, so consider the $G_i$.  For fixed $i$, if Burner plays some base vertex $v$ outside of $G_i$, then fire can spread from $v$ to $G_i$ only by way of the clause vertices.  The distance from any clause vertex to vertex $a_i^{4n-2i+1}$ is at least $8n-4i+2$ which, in turn, is at least $4n+2$.  Hence, even if Burner plays a clause vertex in round 1, the fire will not reach $a_i^{4n-2i+1}$ until after round $4n+1$.  Thus, in order for the game to end within $4n+1$ rounds, Burner must play at least one move in $G_i$.  

Because Burner must play at least one move in each $G_i$ and $H_i$, and because he does not play $z$ except perhaps on his last turn, it follows that Burner's first $2n$ moves are base vertices in the $G_i$ and $H_i$; consequently, Staller's first $2n$ moves are all pendant vertices with already-burned neighbors.  Thus, vertex $z$ is not burned within the first $4n$ rounds of the game, so Burner must play $z$ in round $4n+1$.  It now follows that Burner plays \textit{exactly} one move in each $G_i$ and $H_i$.  Note that $G_i$ has radius $4n-2i+1$, with the only central vertices being $x_i^T$ and $x_i^F$.  In particular, $G_1$ has radius $4n-1$.  Burner must cause all base vertices of $G_1$ to burn by the end of round $4n$ (so that the pendant vertices burn during the spreading phase of round $4n+1$), so he must play either $x_1^T$ or $x_1^F$ in round 1.  Similarly, he must play either $x_2^T$ or $x_2^F$ in round 3, either $x_3^T$ or $x_3^F$ in round 5 and, in general, either $x_i^T$ or $x_i^F$ in round $2i-1$ for all $i \in \{1, \dots, n\}$.  Thus Burner cannot play in $H_i$ until, at the earliest, round $2n+1$.  Since $H_1$ has radius $2n-1$, in order to cause all base vertices of $H_1$ to burn by the end of round $4n$, Burner must play a central vertex of $H_1$ in round $2n+1$.  Similar reasoning shows that Burner must play a central vertex of $H_i$ in round $2(n+i)-1$ for all $i \in \{1, \dots, n\}$.

It follows from the observations above that Burner's strategy will cause all base vertices in the $G_i$ and $H_i$ to burn by the end of round $4n$ (which in turn causes all yet-unburned pendant vertices to burn in round $4n+1$).  Additionally, Burner will burn vertex $z$ in round $4n+1$.  Thus, all vertices will burn except perhaps for the clause vertices and their attached pendants.  As noted above, if Burner plays $x_i^T$ in round $2i-1$, then in round $4n$, the fire from $x_i^T$ spreads to all clause vertices adjacent to $t_i^{4n-2i}$, but not to those adjacent to $f_i^{4n-2i}$.  Likewise, playing $x_i^F$ ensures that the clause vertices adjacent to $f_i^{4n-2i}$ burn in round $4n$, rather than those adjacent to $t_i^{4n-2i}$.  We have already seen that if Burner's moves in the $G_i$ correspond to a satisfying assignment for $\varphi$, then all of $G$ burns within $4n+1$ rounds, so suppose that Burner's moves do not correspond to a satisfying assignment.  Let $C$ be a clause of $\varphi$ not satisfied by the truth assignment corresponding to Burner's moves.  After the spreading phase of round $4n$, the clause vertices $v_C$ and $v'_C$ have not yet burned.  Since these vertices have not burned, neither have their attached pendant vertices; in particular, neither of the pendants can have been played by Staller.  Because the pendants attached to $v_C$ and $v'_C$ are not adjacent and have no common neighbor, no matter what Staller plays in round $4n$, at least one pendant will remain unburned after the spreading phase of round $4n+1$.  Since Burner must play $z$ in round $4n+1$, it follows that at least one of the pendants will remain unburned at the end of round $4n+1$.

We have argued that Burner may play so as to end the game within $4n+1$ rounds if, and only if, $\varphi$ has a satisfying assignment.  Since $G$ has size polynomial in $n$, this construction provides a polynomial-time reduction from 3-SAT to \textsc{BurningGame}.  Since 3-SAT is NP-hard, so must be \textsc{BurningGame}.
\end{proof}

While Theorem \ref{thm:NP_hard} shows that \textsc{BurningGame} is NP-hard, we make no attempt to show that it is NP-complete; in fact, we suspect that it is not. The burning game is a two-player game whose length is polynomially bounded.  As such, \textsc{BurningGame} clearly belongs to PSPACE -- it can be solved in polynomial space using, for example, a standard backtracking algorithm -- and it is natural to suspect that \textsc{BurningGame} is in fact PSPACE-complete.  Intuitively, the construction in Theorem \ref{thm:NP_hard} uses pendant vertices to ensure that Staller can always play a ``useless'' move; thus, Staller's impact of the game is essentially eliminated, and we are effectively left with a one-player game.  Any proof that \textsc{BurningGame} is PSPACE-hard would need to treat the burning game as a ``true'' two-player game by using a reduction wherein Staller's optimal strategy is less straightforward and her choices are more impactful.

\section{Further directions}
\label{sec:questions}

While this paper contains several results about the burning game, including some tools that might be useful in the future research, there are many interesting questions left open. In a sequel paper we will provide Nordhaus-Gaddum-type results and bounds on the game burning numbers of graph products (some of which appear in the work-in-progress~\cite{arxiv-burning}).  Below, we suggest some additional possible directions for further research.\\

As mentioned in Section \ref{sec:upper-bound}, $\bg(G) \leq \min \{ \bg(T) \colon T \text{ is a spanning tree of }G \}$, but the equality is not attained for every graph $G$. While $\bg(T) - \bg(G)$ for some spanning tree $T$ of $G$ can be arbitrarily large (for example, take $G = K_n$ and $T = P_n$), it is not clear how large $\min \{\bg(T) - \bg(G)\}$ can be, where the minimum is over all spanning trees $T$ of $G$.

\begin{question}
    Is it true that $\min \{\bg(T) - \bg(G) \colon T \text{ is a spanning tree of }G \}$ is bounded from above by a constant?
\end{question}

As mentioned in Section \ref{sec:upper-bound}, we were not able to fully determine $b_g(P_n)$ (though our lower and upper bound differ by at most 1). Based on computational evidence, Janko Gravner \cite{gravner2024} suggested the following.

\begin{question}
    Is it true that for $n \geq 1$ we have $$b_g(P_n) = \begin{cases}
        \ceil{\sqrt{2n+1}-1}; \text{ if } n = \frac{k(k+1)}{2}+1 \text{ for some } k \geq 2, \\
        \ceil{\sqrt{2n+\frac{1}{4}}-\frac{1}{2}}; \text{ otherwise?}        
    \end{cases}$$
\end{question}

It would also be of interest to determine exact values of $b_g'(P_n)$, $b_g(C_n)$ and $b_g'(C_n)$.

In Section \ref{sec:upper-bound}, we showed that, assuming the truth of the burning number conjecture, paths have the asymptotically largest game burning number among $n$-vertex connected graphs.  This result suggests the following natural question:

\begin{question}
    \label{que:game-burning-number-conjecture}
    Is $\bg(G) \leq \bg(P_n)$ for every connected graph $G$ on $n$ vertices?
\end{question}

It would also be interesting to know whether the burning number conjecture implies Question \ref{que:game-burning-number-conjecture} or vice versa. 

Finally, Section \ref{sec:complexity} suggests the following question:

\begin{question}
    Is \textsc{BurningGame} PSPACE-complete?
\end{question}
\section*{Acknowledgement}
The work was initiated at the 2nd Workshop on Games on Graphs on Rogla in June 2024. 

N.C. acknowledges partial support by the Slovenian Research and Innovation Agency (I0-0035, research program P1-0404, and research projects N1-0210, N1-0370, J1-3003, and J1-4008).

V.I.C.\ acknowledges the financial support from the Slovenian Research and Innovation Agency (Z1-50003, P1-0297, N1-0218, N1-0285, N1-0355) and the European Union (ERC, KARST, 101071836).

M.J.\ acknowledges the financial support of the Slovenian Research and Innovation Agency (research core funding No.\ P1-0297 and projects N1-0431, N1-0285, J1-3002, and J1-4008).

M.M.\ acknowledges the partial financial support of the Provincial Secretariat for Higher Education and Scientific Research, Province of Vojvodina (Grant No.~142-451-2686/2021) and of Ministry of Science, Technological Development and Innovation of Republic of Serbia (Grants 451-03-137/2025-03/200125 \& 451-03-136/2025-03/200125).

\end{document}